\documentclass[a4paper,10pt]{amsart}
\usepackage[english]{babel}

\usepackage{amsmath}
\usepackage{a4wide}
\usepackage{amssymb}
\usepackage{amsfonts}
\usepackage{enumerate}
\usepackage{amscd}
\usepackage{mathrsfs}
\usepackage{amsthm}
\usepackage[all]{xy}
\usepackage[colorlinks]{hyperref}
\usepackage{tikz-cd}
\usepackage{bbm}
\usepackage[mathscr]{euscript}
\usepackage[
textwidth=3cm,
textsize=small,
colorinlistoftodos]
{todonotes}

\newcommand{\pr}{\mathrm{pr}}

\newcommand{\Hom}{\operatorname{Hom}}
\DeclareMathOperator{\Ho}{Ho}

\DeclareMathOperator{\cQ}{\mathcal{Q}}
\newcommand{\cW}{\mathcal{W}}
\newcommand{\cR}{\mathcal{R}}

\newcommand{\cB}{\mathcal{B}}

\newcommand{\cA}{\mathcal{A}}

\newcommand{\cM}{\mathcal{M}}
\newcommand{\cC}{C}
\newcommand{\cD}{D}
\newcommand{\cN}{\mathcal{N}}
\newcommand{\cL}{\mathcal{L}}
\newcommand{\cZ}{\mathcal{Z}}

\newcommand{\lra}{\rightarrow}

\newcommand{\tZ}{\widetilde{Z}}
\newcommand{\tB}{\widetilde{B}}

\newcommand{\bbZ}{\mathbb{Z}}
\newcommand{\Xb}[1]{X_{\beta_{#1}}}
\newcommand{\Xa}[1]{X_{\alpha_{#1}}}
\newcommand{\Xz}[1]{X_{\zeta_{#1}}}
\newcommand{\Yb}[1]{Y_{\beta_{#1}}}
\newcommand{\Ya}[1]{Y_{\alpha_{#1}}}
\newcommand{\Yz}[1]{Y_{\zeta_{#1}}}
\newcommand{\Eg}[1]{E_{\gamma_{#1}}}

\newcommand{\Ez}[1]{E_{\zeta_{#1}}}
\newcommand{\Eb}[1]{E_{\beta_{#1}}}

\newcommand{\tA}{\widetilde{A}}
\newcommand{\tk}{\widetilde{k}}

\newcommand{\prog}[1]{\pr/\gamma_{#1}}
\newcommand{\proz}[1]{\pr/\zeta_{#1}}
\newcommand{\E}{E(1)_*}
 \DeclareMathOperator\coker{coker}

\newcommand{\ie}{i.e.}
\newcommand{\eg}{e.g.}
\newcommand{\cAp}{\mathcal{A}_{\mathrm{proj}}}

\newcommand{\opp}{\mathrm{op}}
\newcommand{\bbL}{\mathbb{L}}

\DeclareMathOperator{\pro}{pr}

\newcommand {\Oplus}{\bigoplus\limits}
\newcommand{\boxprod}{\mathbin\square}

\DeclareMathOperator{\Fun}{Fun}
\newcommand{\Dopp}{\Delta^{\opp}}
\DeclareMathOperator{\Map}{Map}

\DeclareMathOperator{\Ch}{Ch}

\newcommand{\sSet}{s\mathrm{Set}}
\newcommand{\Set}{\mathrm{Set}}
\newcommand{\Sp}{\mathrm{Sp}}

\newcommand{\x}{\times}
\newcommand{\ox}{\otimes}

\newcommand{\twc}{\mathrm{C}^{([1],1)}(\cA)}
\newcommand{\dtwc}{\mathrm{D}^{([1],1)}(\cA) }

\DeclareMathOperator{\Lan}{Lan}

\DeclareMathOperator*{\colim}{colim}
\DeclareMathOperator*{\hocolim}{hocolim}
\DeclareMathOperator*{\hocofib}{hocofib}
\DeclareMathOperator*{\holim}{holim}

\DeclareMathOperator{\cone}{cone}
\DeclareMathOperator{\srep}{srep}

\DeclareMathOperator{\cofib}{cofib}

\numberwithin{equation}{subsection}
\newtheorem{thm}[equation]{Theorem}
\newtheorem{lemma}[equation]{Lemma}
\newtheorem{prop}[equation]{Proposition}
\newtheorem{cor}[equation]{Corollary}
\theoremstyle{definition}
\newtheorem{defn}[equation]{Definition}
\newtheorem{conv}[equation]{Convention}
\newtheorem{remark}[equation]{Remark}

\newtheorem{example}[equation]{Example}

\begin{document}

\title[]
{Monoidal Properties of Franke's Exotic Equivalence}

\author{Nikitas Nikandros}
\address[N. Nikandros]{University of Kent \\ School of Mathematics, Statistics and Actuarial Science\\ Sibson building \\ Canterbury, CT2 7FS, UK}
\email{nnikandros@gmail.com}

\author{Constanze Roitzheim}
\address[C. Roitzheim]{University of Kent \\ School of Mathematics, Statistics and Actuarial Science\\ Sibson building \\ Canterbury, CT2 7FS, UK}
\email{c.roitzheim@kent.ac.uk}

\begin{abstract}
Franke's reconstruction functor $\cR$ is known to provide examples of triangulated equivalences between homotopy categories of stable model categories, which are exotic in the sense that the underlying model categories are not Quillen equivalent. We show that, while not being a tensor-triangulated functor in general, $\cR$ is compatible with monoidal products.
\end{abstract}

\maketitle

\section{Introduction}

For several decades, Franke's exotic equivalence has been fascinating to homotopy theorists, as it is a rare example of a machinery that provides an equivalence up to homotopy between two model categories which are not Quillen equivalent. In practice, the known situations where Franke's construction can be applied to obtain the equivalence 
\[
\cR\colon \dtwc \longrightarrow \Ho(\cM)
\]
link an algebraic model category ($\dtwc$ is the derived category of a flavour of chain complexes in a suitable abelian category $\cA$) with a stable model category $\cM$ which is not necessarily algebraic. 
Key examples include 
\begin{itemize}
\item $\cA$ the category of $\pi_*(R)$-modules for a ring spectrum $R$ and $\cM$ the category of modules over $R$, together with some extra assumption on the projective dimension of $\pi_*(R)$ as well as $\pi_*(R)$ being concentrated in degrees that are multiples of some $N>1$, 
\item $\cA$ the category of $E(1)_*E(1)$-comodules and $\cM$ the category of $K$-local spectra at an odd prime.
\end{itemize}
In this paper, we will always assume that $\cR$ exists and is an equivalence.

Both the algebraic side $\dtwc$ and the topological side $\Ho(\cM)$ are equipped with monoidal structures derived from the monoidal model category structures on $\twc$ and $\cM$, so it is only natural to consider whether $\cR$ is compatible with these. But as $\cR$ is not derived from a Quillen functor $\twc \longrightarrow \cM$, this problem requires a different approach working closely with the construction of $\cR$ itself. 

The example of $K$-local spectra at $p=3$ tells us that we cannot expect $\cR$ to be a monoidal functor in general: the preimage of the mod-3 Moore spectrum is a chain complex that is a monoid, whereas the mod-3 Moore spectrum has no associative multiplication \cite[Remark 1.4.2]{Nora}.
However, we obtain the following, which is the main result of this article.

\begin{thm}
Let $(\cM,\wedge)$ be a simplicial stable monoidal model category and let $(\cA,\otimes)$ be a hereditary abelian monoidal category with enough projectives such that
Franke's reconstruction functor $\cR$ exists and is an equivalence. Then 
\[ \cR\colon (\dtwc,\otimes^{\bbL}) \lra (\Ho(\cM),\wedge^{\bbL}) \]
commutes with the respective monoidal products up to a natural isomorphism
\[ \cR(M_*\otimes^{\bbL} N_*) \cong \cR(M_*)\wedge^{\bbL} \cR(N_*).   \]
\end{thm}

The reconstruction functor $\cR$ rebuilds $\cM$ from algebraic data in the following way. Firstly, part of the assumptions on $\cA$ is that it splits into shifted copies of a smaller abelian category $\cB$. This is then used to split an object of $\twc$ into pieces, which are placed in certain crown-shaped diagram $\cC_N$. Using this piecewise data, one then constructs a $C_N$-shaped diagram in $\cM$. Finally, the homotopy colimit over $C_N$ is applied to get to $\Ho(\cM)$. Specifically, $\cR$ is the composite
\[
\cR: \dtwc \xrightarrow{\cQ^{-1}} \cL \subseteq \Ho(\cM^{\cC_N}) \xrightarrow{\hocolim_{\cC_N}} \Ho(\mathcal{M}).
\]
We therefore take the diagram
\[
\xymatrix{  \dtwc \times \dtwc \ar[d]_{\otimes^{\bbL}} \ar[rr]^{\cR \wedge^{\bbL} \cR} & & \Ho(\cM) \times \Ho(\cM) \ar[d]_{\wedge^\mathbb{L}} \\ 
\dtwc \ar[rr]^{\cR} & & \Ho(\cM)
}
\]
which we would like to show to be commutative
and refine it in the way below in order to deal with the different components of $\cR$ separately.

\begin{eqnarray}\label{bigdiagram}
\xymatrix{ & & & & \\
 \dtwc \times \dtwc \ar[ddd]_{-\otimes^\mathbb{L} - }  \ar@<0.5ex>[rr]&   & \ar[rr]  \Ho(\cM^{\cC_N})\times \Ho(\cM^{\cC_N}) \ar@<0.5ex>[ll] \ar[d]_{\wedge^\mathbb{L}}^{\cong} & & \Ho(\cM) \ar@{=}[ddd] \\
& & \ar[d]_{\mathbb{L}\pro_!=\Ho\Lan_{\pro}}  \Ho(\cM^{\cC_N\times \cC_N}) \ar[urr]^{\hocolim} & & \\
& & \ar[d]_{i^*} \Ho(\cM^{\cD_N}) \ar[uurr]_{\hocolim_{\cD_N}} & & \\
\dtwc   \ar@<0.5ex>[rr]^{\mathcal{Q}^{-1}}  & &\Ho(\cM^{\cC_N}) \ar@<0.5ex>[ll]^{\mathcal{Q}} \ar[rr]_{\hocolim_{\cC_N}} & & \Ho(\cM)
}
\end{eqnarray}

Here, $\cD_N$ denotes a suitable modification of the crown-shaped diagram $\cC_N$ together with an inclusion $i: \cC_N \longrightarrow \cD_N$. (All the ingredients will of course be defined in detail where appropriate.) 
The outline of our proof roughly follows the key points of \cite{Nora}; however, we choose to work in a setting of model categories, which makes our exposition more explicit and straightforward. 
Overall, we have relied on contemporary methods and we refer to modern literature. Our techniques put Ganter’s theorem in firm rigorous footing and in better context with other existing literature, as well as hopefully making it more adaptable to future generalisations.

This paper is organised as follows.

In Section \ref{sec:prelims} we recall the background and tools that we need for our main result and proof, namely simplicial replacements, homotopy Kan extensions, monoidal structures on diagram model categories, a specific mapping cone construction, calculating homotopy colimits using the homology of a category with coefficients in a functor, as well as a recap of the construction of Franke's functor.

In Section \ref{sec:monoidalQ} we will begin by setting up one of our main results, which involves working out the middle vertical part of the diagram \ref{bigdiagram}. The key ingredient is given by a spectral sequence argument calculating the vertices of the functor $\mathbb{L}\pro_!(X \wedge Y)$. We will then feed this into the definition of Franke's functor $\cQ$ in order to obtain the necessarily formulas for monoidality on the left hand side of \ref{bigdiagram}, dealing with underlying graded modules of the twisted chain complexes and the differentials separately. 

Section \ref{sec:mainresult} now wraps up the right hand side of the diagram \ref{bigdiagram} which mostly involves standard properties of homotopy colimits. We can finally assemble these results into the proof of the main theorem and finish with some examples.

\section*{Acknowledgments}
This paper is based on the PhD thesis of the first author under the supervision of the second author. We would like to acknowledge EPSRC grant EP/R513246/1 for funding this project.
Furthermore, we would like to thank Nora Ganter, Irakli Patchkoria and Neil Strickland for helpful comments and support.

\section{Preliminaries}\label{sec:prelims}

In this section we will introduce some of the terminology that we need for our result. We assume that the reader is familiar with the basic background regarding simplicial sets, homological algebra and model categories.

The category of simplicial sets is denoted by $s\Set$. For $n\geq 0, \Delta^n $ denotes the standard $n$-simplex. For an arbitrary category $\mathcal{C}$, the notation $s\mathcal{C}$ stands for the simplicial objects in $\mathcal{C}$, i.e., $s\mathcal{C}=\Fun(\Dopp,\mathcal{C})$. We have $I =\Delta^1$ and $I_+ = \Delta^1 \cup \ast$ and $S^0= \Delta^0 \cup \ast.$ Similarly, $S^1$ stands for the simplicial circle $I/(0\sim 1)$, that is, $\Delta^1/\partial \Delta^1$.

We will let $\cA$ be a graded ($\bbZ$-graded) abelian category, which means that $\cA$
possesses a shift functor $[1]$ which is an equivalence of categories, and
$[n]$ denotes the $n$-fold iteration of $[1]$. The \emph{graded global homological dimension} of $\cA$, $\operatorname{gl.dim} \cA$, is the supremum of the projective dimensions of objects in $\cA$. An abelian category $\cA$ is called \emph{hereditary} if $\operatorname{gl.dim} \cA=1$. There are other, equivalent descriptions of hereditary abelian categories but this one suits our purposes best.

\subsection{Model Categories}

We will now set up our background on model categories. We write any cofibrant replacement functor $Q\colon \cM\lra \cM$ that comes with a natural weak equivalence $q\colon Q\lra 1_{\cM}$.
\begin{conv}\label{homotopycategoryofM}
	We let $\Ho(\cM)$ denote the category $\cM_{\mathrm{cof}}[\cW^{-1}]$ where $\cM_{\mathrm{cof}}$ denotes the full subcategory of cofibrant objects of $\cM$, and we denote the set of morphisms in $\Ho(\cM)$ by $[X,Y].$
\end{conv}
Convention \ref{homotopycategoryofM} allows us to provide a very simple description of the \emph{left derived functor} $\bbL F$ of a left Quillen functor $F: \cM \longrightarrow \cN$. Indeed, the functor
\[F\left.\right|_{\cM_{\text{cof}}}\colon \cM_{\text{cof}}\lra \cN_{\text{cof}}    \]
preserves weak equivalences and, therefore, it induces a functor between the localizations. This functor is precisely $\bbL F$ with our convention.

Finally, an important class of model categories is the class of \emph{simplicial model categories}. These are model categories which are enriched, tensored and cotensored over $\sSet$ and which satisfy the pushout-product axiom (SM7). If a 
simplicial model category is pointed, \ie, the terminal object is isomorphic to the initial one, then $\cM$ is enriched over the category $\sSet_*$ of pointed simplicial sets. In particular, we have functors
\[ -\otimes-\colon \sSet_*\times \cM\lra \cM, \ \ \  \Map_{\cM}(-.-):\cM^{\opp}\times \cM\lra \sSet_*, \]
and the adjunction
\[\Hom_{\cM}(K\wedge X,Y)\cong \Hom_{\sSet}(K,\Map_{\cM}(X,Y)),   \]
see \cite[Definition 6.1.28]{BRfound} and \cite[Section 11.4]{Riehl}.

\subsubsection{Diagram Categories}
We will use model structures on diagram categories throughout the paper. Below we introduce the definition of a direct category which is a generalization of the concept of a poset, see \cite[Definition 5.1.1]{Hovey} for further details. . 
\begin{defn}\label{direct}
	Let $\omega$ denote the poset category of the ordered set $\left\{0,1,2,\ldots \right\}$. A small category $J$ is called \emph{direct} if there is a functor $f\colon J\lra \omega$ that sends non-identity morphisms to non-identity morphisms. We refer to $f(j)$ as the \emph{degree} of the object $j$. Dually, $J$ is an \emph{inverse} category if there is a functor $J^{\opp} \to \omega$ that 
	sends non-identity morphisms to non-identity morphisms
\end{defn}
Any finite poset $J$ is a direct category, and dually $J^{\opp}$ is an inverse category. We provide some examples that will be useful later on.
\begin{defn}\label{latchdirect}
	Suppose $\cM$ is a small category with small colimits, $J$ a small category, $z$ an object in $J$ and $J_z$ the category of all non-identity morphisms with codomain $z$. The \emph{latching space 
		functor } $L_z\colon \cM^{J}\lra \cM$ is the composition
	\[ \cM^{J}\lra \cM^{J_z} \xrightarrow{\colim} \cM,    \]
	where the first arrow is the restriction functor. 
	Equivalently the latching space of a diagram $X$ is given by
\[L_zX= \colim\left( J_z \hookrightarrow J \xrightarrow{X} \cM  \right),  \]
where $J_z \hookrightarrow J $ is the inclusion. 
	\end{defn}
Note that we have a natural transformation
	$ L_zX\lra X_z $
	for any fixed object $z\in J$.

We can now describe the \emph{projective model structure} on $\cM^J$, see \cite[Theorem 5.1.3]{Hovey}.
\begin{prop}\label{modelondiagrams}
	Given a model category $\cM$ and a direct category $J$, there is a model structure on $\cM^{J}$ in which a morphism $f\colon X\lra Y$ is a weak equivalence (resp. fibration) if and only if the map $f_z\colon X_z\lra Y_z$ 
	is a weak 
	equivalence (resp. fibration) for all $z\in J$. Furthermore, $f\colon X\lra Y$ is an (acyclic) cofibration if and only if the induced map 
	\begin{equation*}
		X_z\coprod_{L_zX}L_zY\lra Y_z
	\end{equation*} 
	is an (acyclic) cofibration for all $z\in J$.
\end{prop}

We will now give the finite posets $J$ that are going to play a central role throughout this paper.

\begin{example}\label{interval}
	By $[1]$ we denote the poset $0\leq 1$. We are aware that early in this section we also denoted the shift functor on graded objects. Both are standard notation, and from our use of the poset $0 \leq 1$ there is vanishingly little danger of confusing those two.
\end{example}

\begin{example}\label{prepushout}
	Consider the poset
	\[\xymatrix{  (0,0)\ar[r] \ar[d] & (1,0) \\
		(0,1)           &  }\]
	denoted by $\ulcorner$. Let $\iota\colon [1]\lra \ulcorner$ be the map of posets which sends $0$ to $(0,0)$ and $1$ to $(1,0)$. In other words, $\iota$ includes the interval $[1]$ to the top horizontal line. Furthermore, consider the product of the interval posets $[1]\times [1]$. It is the following poset
	\[\xymatrix{ (0,0) \ar[d] \ar[r] & (1,0)\ar[d] \\
		(0,1)  \ar[r]      & (1,1)  }     \]	
	and we let $i_{\ulcorner}\colon \ulcorner\lra [1]\times [1]$ the inclusion. 					
\end{example}

\begin{example}\label{CN}
Let $N \ge 2$ be a natural number. The poset $\cC_N$ consists of elements $\left\{\beta_i,\zeta_i\left|\right. i\in\bbZ/N\bbZ\right\}$ such that $\beta_i<\zeta_i$ and $\beta_{i}<\zeta_{i+1}$ for $i\in \bbZ/N\bbZ$, \ie, 
\begin{equation*}\xymatrix{  \zeta_0  &  \zeta_1  & \ldots & \zeta_{N-1}   \\ 
              \beta_0 \ar[u] \ar[ur] & \ar[u]  \ar[ur] \beta_1   & \ar[ur] \ldots & \ar[u] \ar[ulll] \beta_{N-1}.       }
							\end{equation*}
Then $X\in \cM^{\cC_N}$ is  cofibrant if and only if the canonical map
$L_z X\lra X_z$
is a cofibration in $\cM$, \ie, if and only if the $\Xb{i}, \Xz{i}$ are cofibrant and the induced morphism $$\Xb{i-1}\vee \Xb{i}\lra \Xz{i}$$ is a cofibration, where $\vee$ is the coproduct in $\cM$. We will refer to an object $X\in \cM^{\cC_N}$ as a \emph{crowned diagram} due to the crown shape of the diagram $\cC_N$. 
\end{example}

\begin{example}\label{DN}
Let $\cD_N$ be the poset consisting of elements $\left\{\beta_n,\gamma_n, \zeta_n \colon n\in \bbZ/N\bbZ  \right\}$ such that $\beta_n \leq \gamma_n \leq\zeta_n$ and 
$\beta_{n} \leq\gamma_{n+1}$ and $\gamma_{n}\leq\zeta_{n+1}$, \ie, 
\begin{equation*}
\xymatrix{\zeta_0  &  \zeta_1  & \ldots & \zeta_{N-1}   \\ 
\gamma_0 \ar[u] \ar[ur] & \ar[u]  \ar[ur] \gamma_1   & \ar[u] \ldots & \ar[u] \ar[ulll] \gamma_{N-1}\\
   \beta_0 \ar[u] \ar[ur] & \ar[u]  \ar[ur] \beta_1   & \ar[u] \ldots & \ar[u] \ar[ulll] \beta_{N-1}   . }
	\end{equation*}

\end{example}

 \begin{remark}
	In what follows, when we have a direct category $I$ and a model category $\cM$, the category of diagrams $\cM^I$ will always have the model structure defined in Proposition \ref{modelondiagrams} without further mention. If not, we 
	will explicitly say so.
\end{remark}

It follows that for any model category $\cM$ and direct category $J$, there is a Quillen adjunction
\[\colim\colon \cM^J \rightleftarrows \cM\colon \operatorname{const}.    \]
(Note that when we write an adjunction, the top arrow will always denote the left adjoint.)

\begin{defn}\label{hocolimderived}
	The left derived functor of $\colim\colon \cM^J\to \cM $ is called the \emph{homotopy colimit} and is denoted by
	\[\hocolim\colon \Ho(\cM^J)\to \Ho(\cM).\]
	If $J=\ulcorner$, then the homotopy colimit is called \emph{homotopy pushout.} A particular example of homotopy pushout is the \emph{homotopy cofiber} which is the homotopy pushout of a diagram of the form
	\[\xymatrix{  X\ar[r]^{f} \ar[d] & Y \\
		\ast           &  }\]
and we write 
\begin{equation}\label{homotopycofiber}
\hocofib(f):=\hocolim(*\leftarrow X \stackrel{f}{\to} Y).
\end{equation}
In general, for notational convenience sometimes a homotopy pushout is denoted by
\[\hocolim(Z \leftarrow X\rightarrow Y):= Z\coprod^h_XY.\]
\end{defn}

\subsubsection{Homotopy Colimits in Simplicial Model Categories}
In Definition \ref{hocolimderived} we recalled the definition of the homotopy colimit as a derived functor. Here, we will present an alternative construction via simplicial techniques. After introducing some definitions we briefly explain how this method provides a good theory of homotopy colimits, see also \cite[Chapters 4,5]{Riehl} and \cite[Section 7]{Shulman}.

Let $\cM$ be a model category and consider the category of simplicial objects $s\cM=\cM^{\Dopp}$. We consider $s\cM$ as a simplicial category with tensors defined objectwise, i.e., for $K\in \sSet$ and $X\in s\cM$ we have
\[(K\otimes X)_n= K\otimes X_n.  \] 

Now, let $\cM$ be a simplicial model category. Given a simplicial object $X\in s\cM$ we can construct an object in $\cM$ via 
\emph{geometric realization}, see \cite[Definition 18.6.2]{Hir1}.
\begin{defn}[Geometric Realization]\label{realization}
	Let $X\in \cM^{\Dopp}$. \emph{The geometric realization} of $X$, denoted as $|X|$ is defined as the coequalizer
	\begin{equation*} 
		\operatorname{coeq}\left(\coprod_{\sigma\colon [n]\lra [k]\in \Delta} \Delta^k \otimes X_n   \rightrightarrows \coprod_{[n]\in \Delta}\Delta^n \otimes  X_n  \right). 
	\end{equation*}
\end{defn}
This is an example of a functor tensor product (coend). In this case, the geometric realization is the functor tensor product of $X\colon \Dopp\to \cM$ and the 
functor $\Delta^{\bullet}\colon \Delta\lra \sSet, [n]\mapsto \Delta^n.$ In other words, the realization $|X|$ is the object
\[\Delta^{\bullet}\otimes_{\Delta^{\opp}}X= \int^{n} \Delta^n \otimes X_n .\]

The following theorem is the cornerstone of our exposition of homotopy colimits using geometric realizations, see \cite[VII 3.6]{GJ}, \cite[18.4.11]{Hir1} and \cite[Corollary 14.3.10]{Riehl}. For details for the Reedy model structure on $s\cM$, see \cite[Definition 2.1]{GJ}.
\begin{thm}\label{realisquillen}
	If $\cM$ is a simplicial model category, then
	\[ \left|-\right| \colon \cM^{\Delta^{\opp}} \to \cM \]
	is a left Quillen functor with respect to the Reedy model structure. In particular, $\left|-\right|$ sends Reedy
	cofibrant simplicial objects to cofibrant objects and preserves objectwise weak equivalences
	between them.
\end{thm}
At this level of generality, this is the strongest result possible. It is not true that geometric realization preserves all objectwise weak equivalences. However, the above will suffice for our purposes. We can now start to work our way to the homotopy colimit of a diagram $X\in \cM^J$ in a simplicial model category $\cM$.

Our first definition towards this goal is the \emph{simplicial replacement functor}. That is to say, given any diagram $F\colon I\lra \cM$ we can replace it with simplicial object in $\cM$ with good 
properties. 
\begin{defn}[Simplicial replacement]\label{simprep}
	Let $I$ be a small category and consider a diagram $X\in\cM^I$. The \emph{simplicial replacement} of $X$ is the simplicial object in $\cM$, denoted $\srep X$ given in simplicial degree $[n]$ by
	\begin{equation*}
		(\srep X)_n= \displaystyle{\coprod_{(i_0\rightarrow i_1\rightarrow \ldots \rightarrow i_n)\in N(I)_n}}X_{i_0}.
	\end{equation*}
\end{defn}
The coproduct is indexed over the set of $n$-chains 
\[\sigma=[i_0\rightarrow i_1\lra \ldots \to i_n] \] 
over the nerve of $I$. 
If $0\leq k<n$, then 
\[d_k\colon (\srep X)_{n}\lra (\srep X)_{n-1}\] maps the term $X_{i_n}$ indexed on $\sigma$ to the term $X_{i_n}$ indexed on 
\[\sigma(k)=[i_0\to i_1\to i_{k-1}\to i_{k+1}\to \ldots \to i_n]\]
via the identity, while for $k=n$, the map $d_n$ sends the term $X_{i_n}$ to $X_{i_{n-1}}$ indexed on 
\[\sigma(n)=[i_0 \rightarrow i_1\rightarrow \ldots \rightarrow i_{n-1}]\] via the induced map 
$X(i_n\rightarrow i_{n-1})$. The degeneracy maps 
\[
s_j\colon (\srep X)_n\lra (\srep X)_{n+1}, 0\leq j\leq n
\]
are easier to define. Each $s_j$ sends the summand $X_{i_n}$ corresponding to the summand
\[ [i_0 \rightarrow i_1\rightarrow \ldots \rightarrow i_n] \]to the identical summand $X_{i_n}$ corresponding to the chain in which one has inserted the identity map $i_j\rightarrow i_j$.    

In other words, the simplicial replacement is the following simplicial object 
\begin{equation*}
	\xymatrix{ \displaystyle{\coprod_{i_0}} X_{i_0}  &  \ar@<.5ex>[l] \ar@<-.5ex>[l]  \displaystyle{\coprod_{i_0\rightarrow i_1}} X_{i_0}  & \ar@<1ex>[l] \ar[l] \ar@<-1ex>[l]   \displaystyle{\coprod_{i_0
				\rightarrow i_1 \rightarrow i_2}} X_{i_0} \cdots  }
\end{equation*}
where degeneracy maps are omitted. Note that this is can also be foudn in literature as the \emph{simplicial bar construction} or \emph{Bousfield-Kan construction} denoted by $B(\ast,I,X)$.

\begin{remark}
	The colimit of a diagram $X \in \cM^{I}$, if it exists, agrees with the colimit of $\srep(X)\in s\cM.$ Indeed, consider the colimit of the diagram $\srep(X)$ as the coequilizer
	\[\coprod_{i}X_i \leftleftarrows \coprod_{j\leftarrow i}X_i,   \]
	but this is precisely the colimit of $X$. Therefore in this case, $\srep(X)$ has the augmentation
	\[\srep(F)\lra \colim_{I}F   ,\]
	where we regard the object $\colim_{I} F$ as a constant simplicial object.
\end{remark}
We therefore reach the following result.
\begin{lemma}\label{lem:realization}
	Given a diagram $X\in \cM^I$ and its simplicial replacement $\srep(X)\in \cM^{\Dopp}$, there is a canonical isomorphism
	\[\colim_{I} X \cong \colim_{\Delta^{\opp}}\left(\srep(X)\right).   \]
\end{lemma}
The proof can be found in \cite[Lemma 4.4.2]{Riehl}. The following lemma will also be of importance, see \cite[Lemma 5.1.2]{Riehl}, \cite[Lemma 8.7]{Shulman}.
\begin{lemma}\label{objectwisecofibrantreedy}
	Let $I$ be a small category and let $\cM$ be a simplicial model category. If $F\in \cM^I$ is objectwise cofibrant, then $\srep(F) \in s\cM$ is Reedy cofibrant.
\end{lemma}
The above Lemma \ref{lem:realization} and Theorem \ref{realisquillen} essentially mean that geomemetric realization of objectwise cofibrant diagrams is a good model for calculating homotopy colimits. For details see \cite[Theorem 6.6.1]{Riehl}.

\subsection{Homotopy Kan Extensions}
In this subsection we will introduce \emph{homotopy Kan extensions}, the homotopy invariant version of ordinary Kan extensions, see e.g. \cite[Section 11.9]{Hir1}. 

Now, let $\cM$ be a model category. Furthermore, let $I, J$ be direct categories and $f\colon I\lra J$ a functor. The pullback functor 
\[f^*\colon \cM^{J}\lra \cM^{I}  \]
preserves weak equivalences, so it defines a functor between homotopy categories, which we denote by the same letter. Recall the functor $\Lan_f=f_!$, left adjoint to $f^*$.
We have the following proposition.
\begin{prop}\label{leftkanrightkanquillenpair}
	Let $\cM$ be a model category and let $f\colon I\lra J$ be a map of direct categories. Then the adjunction
	\[ f_!\colon \cM^{I}\rightleftarrows \cM^{J}\colon f^*   \]
	is a Quillen adjunction. 
\end{prop}
\begin{proof}
	This follows from the definition of the projective model structure, see Proposition \ref{modelondiagrams}. The functor $f^*$ is a right adjoint by construction. It preserves weak equivalences and projective fibrations, which means that $f^*$ is also a right Quillen functor.
\end{proof}
Thus, the derived functors of the adjoint pair $(f_!,f^*)$ define an adjoint pair on the level of homotopy categories
\[  \bbL\mathrm{Lan}_f:= \bbL f_! \colon \Ho(\cM^I) \rightleftarrows  \Ho(\cM^J)\colon \mathbb{R}f^*.    \]
A useful fact about homotopy Kan extensions is that they does not change the homotopy colimit of a diagram, which is similar to the properties of ordinary Kan extensions.
\begin{cor}\label{landoesnotchangecolimit}
	Let $\cM$ be a model category, $f\colon I\lra J$ a map of direct categories and let $X\in \cM^{I}$. Then there is a canonical isomorphism in $\Ho(\cM)$ 
	\[\hocolim_{J} \bbL f_!X \cong \hocolim_{I}X.\]
\end{cor}
\begin{proof}
	This follows from the fact that for every pair of left Quillen functors $F$ and $G$ there is a natural isomorphism
	\[\bbL F \circ \bbL G \lra \bbL (F\circ G),\] 
	see \cite[Theorem 1.37]{Hovey}, together with the natural isomorphism
	\[\colim_{J}\Lan_fX\cong \colim_{I}X.\]
\end{proof}

To conclude this section, we will shortly discuss how one calculates the values and edges of a homotopy Kan extension. 
Recall the notion of a \emph{slice category} for given posets $\cC$ and $\cD$ and a functor $f\colon \cC\lra \cD$, namely
\begin{equation}\label{sliceposets}
f/d= \left\{c\in \cC \colon f(c) \leq d\right\}
\end{equation}   
for $d\in \cD$. 
The following is \cite[Proposition 1.14]{CI09}, which tells us that homotopy Kan extensions can be computed pointwise.
\begin{prop}\label{holanformula}
	Let $f\colon I\to J$ be a map of posets  and let $X$ be any functor $I\lra \cM$. For any object $j\in J$ there is a canonical isomorphism in $\Ho(\cM)$
	\[ (\bbL f_!F)_j \cong \hocolim\left( f/j \xrightarrow{\pi}  I \xrightarrow{X} \cM \right) .\]
\end{prop}

\subsection{Monoidal Model Categories}

Let us now turn to some results concerning monoidal model categories, 
see e.g. \cite[Definition 4.2.6]{Hovey}, \cite[Definition 6.1.9]{BR} or \cite[Definition 11.4.6]{Riehl} for definitions. 

\begin{remark}
	Let $(\mathcal{C},\wedge)$ be a closed symmetric monoidal category and let $f\colon X_0\lra X_1$ and $g\colon Y_0\lra Y_1$ be maps in $\mathcal{C}$. The pushout-product map is the universal arrow 
	\[f\boxprod g\colon X_0\wedge Y_1\coprod_{X_0\wedge Y_0} X_1\otimes Y_0\lra X_1\wedge Y_1.  \]
	
	Another way to see the pushout-product map is as a left Kan extension. Again, consider a cocomplete, (closed) monoidal category $(\mathcal{C},\wedge)$. Let $[1]=\left\{0\leq 1\right\}$. Furthermore, consider  the following map of posets.
	\begin{align*}
		\pr\colon [1] \times [1] &\lra [1] \\
		(0,0), (1,0), (0,1)&\mapsto 0 \\
		(1,1)        &\mapsto 1   
	\end{align*}
	Now let $f$ and $g$ be morphisms in $\mathcal{C}$. We can consider them as as objects in the arrow category $f,g\in \mathcal{C}^{[1]}$. The functors $f\colon[1]\lra \mathcal{C}$ and $g\colon [1]\lra \mathcal{C}$ give rise to their objectwise tensor product $f\wedge g$, see Definition \ref{objectwisesmash}. That is, the functor
	\[f\wedge g\colon [1] \times [1] \lra \mathcal{C} \]
	is the following commutative diagram.
	\[\xymatrix{   X_0\wedge Y_0  \ar[r] \ar[d]&  X_1\wedge Y_0  \ar[d]    \\
		X_0\wedge Y_1 \ar[r]  & X_1\wedge Y_0  .}\]
	Note that the slice category $\pr/0$ is the poset $\ulcorner$ and the slice $\pr/1$ is the whole square. It follows that the map
	\[\colim_{\ulcorner}\left(f\wedge g\right)\to \colim_{[1] \times [1]} \left(f\wedge g\right)    \]
	induced by the inclusion $\ulcorner \hookrightarrow [1] \times [1]$ is exactly the map 
	\[f\boxprod g\colon X_0\wedge Y_1 \coprod_{X_0\wedge Y_0} X_1\wedge Y_1\lra X_1\wedge Y_1.\]
	So indeed, $(\Lan_{\pr}(f\wedge g))= \pr_!(f\wedge g)= f\boxprod g$.
\end{remark}

\subsubsection{Smash Products for Diagram Categories}
A monoidal category $(\cM,\wedge)$ gives rise to more monoidal categories by considering diagrams from small categories into $\cM$. In our next example we discuss how this is related to model category theory.

\begin{defn}\label{objectwisesmash} 
Let $(\cM,\wedge)$ be a monoidal category and let $I$ and $J$ be direct categories. We define the external product, which is the bifunctor
\[-\wedge-\colon \cM^{I} \times \cM^{J}\lra \cM^{I\times J}\]
sending $(X,Y)$ to the diagram 
\[X\wedge Y\colon I\x J\lra \cM, \ \ (i,j)\mapsto X_i\wedge Y_j.\]
\end{defn}

The external product is part of a two-variable adjunction. Since we do not use the extra structure will not define the other two functors in the two-variable adjunction. We have the following proposition.

\begin{prop}\label{totalderivedobjectwise}
Let $(\cM,\wedge)$ be a monoidal model category. Then, the bifunctor
\[ -\wedge-\colon\cM^{I}\times \cM^{J}\lra \cM^{I\times J}  \]
is a Quillen bifunctor, that is to say, it has a total left derived functor
\[ -\wedge^{\bbL}-\colon \Ho(\cM^{I})\times \Ho(\cM^{J})\lra \Ho(\cM^{I\times J}).  \]
\end{prop}

\begin{proof}
Suppose that the \emph{injective} model structures $\cM_{\mathrm{inj}}^{I}, \cM_{\mathrm{inj}}^{J}$ and  $\cM_{\mathrm{inj}}^{I\x J}$ exist, \eg, if 
$\cM$ is a combinatorial model category. Since in the injective model structures the cofibrations are the objectwise cofibrations, the above proposition follows directly. The universal property of $-\wedge^{\bbL}-$ implies that up to canonical isomorphism both constructions give the same result.
\end{proof}


We have the following  corollary.
\begin{cor}\label{hocolimIJ}
In the context of Proposition \ref{totalderivedobjectwise}, there is a functor isomorphism
\[\hocolim_{I\times J} (X\wedge^{\bbL}Y)\cong (\hocolim_{I}X)\wedge^{\bbL} (\hocolim_{J}Y).   \]
\end{cor}
\begin{proof}
From Proposition \ref{totalderivedobjectwise}, it follows that the external product preserves diagram cofibrant objects and preserves trivial diagram cofibrations between diagram cofibrant objects. The result now follows from the strict formula
\[ \colim_{I\times J}(X\wedge Y)\cong (\colim_I X) \wedge (\colim_JY)  \]
as all the objects involved are cofibrant.
\end{proof}

As a consequence of Proposition \ref{totalderivedobjectwise}, we also obtain the following.

\begin{example}\label{diagramsismonoidal}
	Let $(\cM,\wedge)$ be a monoidal model category and let $J$ be a direct category. Consider the diagram category $\cM^J$ with the model structure \ref{modelondiagrams}. The category $\cM^J$ inherits a monoidal structure
	\[\cM^J\x \cM^J\to \cM^J, \ \ (X,Y) \mapsto X\wedge Y, \]
	where $X\wedge Y$ is the diagram $j\mapsto X_j\wedge Y_j$. By a proof analogous to Proposition \ref{totalderivedobjectwise}, $(\cM^J,\wedge)$ is a monoidal model category.
\end{example}

\begin{cor}\label{conemonoidal}
Let $(\cM,\wedge)$ be a pointed symmetric monoidal model category, and let $f\colon X\lra Y$ and $g\colon U\lra V$ be morphisms in $\cM$. There is a canonical isomorphism
\[\hocofib(f)\wedge^{\bbL} \hocofib(g)\cong \hocofib(f\boxprod^{\bbL} g). \]
\end{cor}
We will provide a proof since it is important to our exposition. A different proof can be found in \cite[Proposition 4.1]{hovey2014smith}.

\begin{proof}
We may assume that $X,Y, U, V$ are cofibrant in $\cM$. By definition,
\[\hocofib(f)\wedge^{\bbL} \hocofib(g)= \hocolim(*\leftarrow X\xrightarrow{f} Y) \wedge^{\bbL} \hocolim(*\leftarrow U\xrightarrow{g} V)  .\]
By Corollary \ref{hocolimIJ}, this is isomorphic to
\begin{equation}\label{double}
	\hocolim
	\left(
	\begin{tikzcd}
		\ast & \arrow[l] X\wedge V \arrow[r] & Y\wedge V \\
		\ast \arrow[u] \arrow[d]   &\arrow[l]  X\wedge U \arrow[d] \arrow[r] \arrow[u]   & Y\wedge U \arrow[u] \arrow[d] \\
		\ast		 & \arrow[l]   \ast \arrow[r]       & \ast
	\end{tikzcd}
	\right).
\end{equation}
We denote the above underlying $\ulcorner\times \ulcorner$-diagram by $\cZ$. We define the following map of posets 
\begin{align*}
	\pr\colon \ulcorner\times \ulcorner& \to \ulcorner \\
	\left((1,0), (1,0)\right)           &\mapsto (1,0) \\
	\left((0,0), (0,0)\right),    \left((0,0), (1,0)\right),   \left((1,0), (0,0)\right) &\mapsto (0,0)\\
	\text{else}                                   &\mapsto (0,1),
\end{align*}
and consider the homotopy left Kan extension 
\begin{equation}\label{holanleftcone}
	\bbL\pr_!\colon \Ho(\cM^{\ulcorner\times \ulcorner})\to \Ho(\cM^{\ulcorner}).
\end{equation}
Applying the formula Proposition \ref{holanformula} to the diagram $\cZ$ we obtain
$(\bbL\pr_!\cZ)_{(1,0)}= Y\wedge V $. Next, for the object $(0,0)$ the slice category $\pr/(0,0)$ is just the poset $\ulcorner$ and we have
\[
(\bbL\pr_!\cZ)_{(0,0)}= 
\hocolim\left( 
\begin{tikzcd}
	X\wedge U \arrow[d,"1\wedge g "] \arrow[r,"f\wedge 1"] & Y\wedge U \\
	X\wedge V                          &
\end{tikzcd}
\right) \]
and finally, $(\bbL\pr_!\cZ)_{(0,1)}\cong \ast$. Note that 
\[(\bbL\pr_!\cZ)_{(0,0)} \to (\bbL\pr_!\cZ)_{(1,0)}= f\boxprod^{\bbL} g.\] Hence, the homotopy left Kan extension \eqref{holanleftcone} of the underlying diagram \eqref{double} is the following $\ulcorner$-diagram.
\[
\begin{tikzcd}
	\displaystyle{X\wedge V \coprod^h_{X\wedge U}} Y\wedge U \arrow[d] \arrow[r] & Y\wedge V \\
	\ast                                            &
\end{tikzcd}\]
It follows directly that the homotopy colimit of this diagram is 
\[\hocofib(f\boxprod^{\bbL} g).\]
\end{proof}

\subsubsection{Stable Model Categories and Triangulated Categories}
Recall that the homotopy category $\Ho(\cM)$ of a pointed model category $\cM$ supports a \emph{suspension} functor 
\[\Sigma\colon \Ho(\cM)\lra \Ho(\cM)  \]
given by 
\[\Sigma X:= \hocolim(\ast \leftarrow X \lra \ast),  \]
with a right adjoint functor
\[\Omega\colon \Ho(\cM)\lra \Ho(\cM)\]
given by 
\[\Omega X= \holim(*\lra X \leftarrow *   ).\]

\begin{defn}\label{stablemodelcat}
	A \emph{stable model category} is a pointed model category for which the functors $\Sigma$ and $\Omega$ are inverse equivalences.
\end{defn}

\begin{example}
The prototypical example of a stable model category is the category of spectra, $\Sp$. There are of course many variants of spectra, but as our result does not depend on a choice of suitable, monoidal model category, we will not need to specify this further.
\end{example}

\begin{example}\label{twistedcomplexes}
Let $\cA$ be a graded abelian category with enough projectives, and let $\twc$ denote the category of \emph{twisted} $([1],1)$-chain complexes or \emph{differential} objects. An object of $\twc$ is a pair $(M_*,d)$ with $M_*\in \cA$ together with a morphism 
(the differential)
\[d\colon M_*\lra M_*[1],\]
such that $d[1]\circ d=0$. 
The category $\twc$ admits a stable model structure, the \emph{projective model structure}, where the weak equivalences are the homology isomorphisms and the fibrations are the surjections. In particular, the cofibrant objects are the projective objects of $\cA$. We let $\dtwc$ denote the homotopy 
category of $\twc$. For an object $(M_*,d)\in \twc$ we define the homology $H(M)= \ker d / \operatorname{im} d  $, and so we have the homology functor
\[H_*\colon \dtwc\lra \cA.\]
In the following we will let $(\cA,\otimes,\mathbf{1})$ be an abelian symmetric monoidal category with enough projectives. In this case $(\twc,\otimes)$ is a monoidal stable model category. Finally, we mention the homology functor $H_*\colon \dtwc\lra \cA$ is a lax symmetric monoidal functor via the K\"unneth morphism.

\bigskip
We note that our methods throughout this paper also work in a setting where $\cA$ does not have enough projectives. In the case of $\cA=E(1)_*E(1)\mbox{-comod}$, $\twc$ can be equipped with a model structure where the cofibrant twisted chain complexes are degreewise projective as $E(1)_*$-modules. This \emph{relative projective model structure} is also monoidal, see \cite[Section 5]{BR}.
\end{example}

If $\cM$ is a pointed simplicial model category, then the suspension functor 
\begin{equation*}
	\Sigma\colon \Ho(\cM)\lra \Ho(\cM)
\end{equation*}
admits a simple description. Indeed, by the simplicial model category axioms, the functor 
\begin{equation*}
	S^1\wedge -\colon \cM\lra \cM
\end{equation*}
defined using the tensor with simplicial sets
is a left Quillen functor. Then, $\Sigma$ can be defined as the left derived functor of $S^1\wedge -$, \ie, 
\[\Sigma X:= S^1\wedge^{\bbL} X= S^1\wedge QX,\] 
see \cite[6.1.1]{Hovey}. Note that if $\cM$ is stable, then the homotopy category $\Ho(\cM)$ is a triangulated category with $\Sigma$ a shift functor, see \cite[Theorem 4.2.1]{BR} and \cite[7.1.6]{Hovey}. 

In a simplicial model category $\cM$ we can choose a particular \emph{model} for the homotopy cofiber \eqref{homotopycofiber} of a morphism, which will help with computations. It is called the \emph{mapping cone} construction. 

\begin{defn}\label{mappingcone}
Suppose $\cM$ is a simplicial stable model category and $f\colon X\lra Y$ a morphism in $\cM_{\mathrm{cof}}$. Let $\cone(f)$ be the pushout of $f$ 
along the canonical morphism 
\[\operatorname{incl}\otimes 1\colon S^0\otimes X\lra (I,0)\otimes X=CX,\] 
that is, $\cone(f)$ comes with the pushout square
\[
\begin{tikzcd}
	X \arrow[r,"f"] \arrow{d}[swap]{\operatorname{incl}\otimes 1} &  Y \arrow{d}\\
	CX \arrow[r]                            & \cone(f).
\end{tikzcd}   \] 
Here $CX=(I,0) \otimes X$ denotes the \emph{cone} of $X$. The natural map 
\[\pi\colon (I,0)\otimes X\lra S^1\otimes X \] and the trivial map \[\ast\colon Y\lra S^1\otimes X\] induce, using the universal property of pushput, a map
$\partial\colon \cone(f)\lra S^1\otimes X.$
\end{defn}
The fact that the mapping cone construction represents the homotopy cofiber and further details can be found in \cite[Section 4.3]{BRfound}. 

\begin{defn}\label{elementary}
Let $\cM$ be a simplicial stable model category and $f\colon X\lra Y$ a morphism in $\cM_{\mathrm{cof}}$. The \emph{elementary triangle} associated to $f$ is the triangle
\begin{equation*}
	X \xrightarrow{f} Y \xrightarrow{\iota} \cone(f) \xrightarrow{\partial} S^1\otimes X.
\end{equation*}
A triangle $(f,g,h)$
\begin{equation*}
	A \xrightarrow{f} B \xrightarrow{g} C \xrightarrow{h} \Sigma A
\end{equation*}
in $\Ho(\cM)$ is called \emph{distinguished} if it is isomorphic to an elementary one.
\end{defn}

\subsection{Homology of a Category with Coefficients in a Functor}\label{sec:homologyofacategory}
In this subsection we will introduce one our main tools, namely homology of a category with coefficients in a functor. It is a particular case of functor homology that assigns the groups 
$\operatorname{Tor}^{I}_*(F,G)$ to functors
$F\colon I\lra \cA$ and $G\colon I^{\opp}\lra \cA$ with $\cA$ an abelian category. Since we do not need such generality, we will introduce it in a more down-to-earth way using simplicial techniques that dates back to Quillen. Traditional references include
\cite{OB} and \cite{OB2}, more contemporary references include \cite{NTG3} and \cite[Chapters 15, 16]{RI20}.

 Before we define the homology of a category with coefficients in a functor we will define the associated complex of a simplicial object 
in an abelian category.
\begin{defn}\label{associatedcomplex}
	Let $D\in s\cA$ be a simplicial object in $\cA$. We define the \emph{associated complex} $(C_{\bullet}(U),\partial)\in \Ch_{\geq 0}(\cA)$ by
	\begin{displaymath}
		C_n(D)=D_n, \ \ \ \  \partial_n=\sum^n_{i=0}(-1)^nd_i\colon C_n(D)\lra C_{n-1}(D).
	\end{displaymath}
Note that the simplicial identities imply $\partial^2 = 0$, so $C_{\bullet}(D)$ is indeed
a chain complex. Moreover, this evidently defines a functor $C\colon s\cA\lra \Ch_{\geq 0}(\cA)$. In other words, the associated complex to a simplicial object $D\in s\cA$ is the following chain complex.
\begin{equation}
		\label{associated}
		D_0 \xleftarrow{d_0-d_1} D_1 \xleftarrow{d_0-d_1+d_2}  D_2 \leftarrow \ldots.
	\end{equation}
\end{defn}
\begin{defn}
	Let $I$ be a small category and consider a diagram $D\colon I\lra \cA$. The \emph{homology of the category} $I$ \emph{with coefficients in the functor} $D$ is defined as the homology of the complex $C_{\bullet}(D)$, \ie, the homology of the 
	associated complex of the simplicial replacement $\srep(D)\in s\cA$. 
\end{defn}
So, unwinding the definition, we start by first taking the simplicial replacement $\srep(D)\colon \Delta^{\opp}\lra \cA$  of $D$, see Definition \ref{simprep}, that is, the diagram 

\begin{equation*}
	\xymatrix{ \displaystyle{\bigoplus_{i_0}} D_{i_0}  &  \ar@<.5ex>[l] \ar@<-.5ex>[l]  \displaystyle{\bigoplus_{i_0\lra i_1}} D_{i_0}  & \ar@<1ex>[l] \ar[l] \ar@<-1ex>[l]   \displaystyle{\bigoplus_
			{i_0\lra i_1 \lra i_2}} D_{i_0} \cdots  .}
\end{equation*}
Then, we consider the associated chain complex \eqref{associated} $C_{\bullet}(D)$. Then we defined $H_p(I;D)$ to be the 
$p$th homology group of the chain complex $C_{\bullet}(D)$.

Now we will investigate how these constructions help us calculate homotopy colimits. First, recall the following.
\begin{defn}
	We call a functor $F_*\colon \Ho(\cM)\lra \cA$ \emph{homological} if it satisfies the following conditions.
	\begin{enumerate}[(i)]
		\item $F_*$ is a graded functor, that is to say, it commutes with suspensions, so there are natural equivalences 
		\[ F_*(\Sigma X)\cong F_*(X)[1]:= F_{*-1}(X) \]
		which are part of the structure.
		\item $F_*$ is additive, \ie, it commutes with arbitrary coproducts.
		\item $F_* $ converts distinguished triangles into long exact sequences.
		\item Furthermore, if $(\cM,\wedge)$ is a monoidal model category and $(\cA,\otimes)$ is a monoidal abelian category, we require that  $F_*$ is lax symmetric monoidal, that is, there is a natural K\"unneth morphism
		\[\kappa_{X,Y}\colon F_*X\otimes F_*Y \lra F_*(X\wedge^{\bbL} Y).   \]
	\end{enumerate}
\end{defn}
 Now let $\cM$ be a  simplicial stable model category, let $I$ be a direct category and let $X\in \Ho(\cM^I)$. Further, let
 \[F_*\colon \Ho(\cM)\lra \cA\]
 be a homological functor into an (graded) abelian category. Then there is a spectral sequence
 \begin{equation}\label{sshomotopycolimit}
E^2_{pq}=H_p(I;F_qX) \Rightarrow F_{p+q}(\hocolim_{J} X),
\end{equation}
see \cite[16.3.1]{RI20}. The construction of the spectral sequence \eqref{sshomotopycolimit} arises from the skeletal filtration of a simplicial object. This spectral sequence will play a central role in our calculations for the monoidal properties of $\cQ$
in Section \ref{sec:monoidalQ}.

\subsection{Franke's Realization Functor}\label{Frankefunctor}
In this subsection we will recall the construction of Franke's equivalence 
\[\cR\colon \dtwc\lra \Ho(\cM) .\]
For a detailed exposition we refer to \cite[Section 3.3]{PA12} and \cite{Ro}.
Recall that $\cC_N$ is the crown-shaped poset from Example \ref{CN}, and that the category $\dtwc$ above is the derived category of twisted chain complexes from Example \ref{twistedcomplexes}, where $\mathcal{A}$ is a graded symmetric monoidal hereditary abelian category with enough projectives, $\mathcal{M}$ is a simplicial stable model category, and $F: \Ho(\cM) \longrightarrow \mathcal{A}$ is a homological functor. Also, we assume $\mathcal{A}$ splits into shifted copies of another abelian category $\mathcal{B}$, $$\mathcal{A} = \bigoplus\limits_{i=0}^{N-1} \mathcal{B}[i]$$ for $N>1$. Under these assumptions, $\cR$ exists and is an equivalence. 

 For an object $X\in \cM^{\cC_N}$ we the structure morphisms of $X$ as follows.
\[l_i\colon \Xb{i}\lra \Xz{i}, \ \ k_i\colon \Xb{i-1}\lra \Xz{i}, \ \ i\in\bbZ/N\bbZ  \]
Furthermore, let 
\begin{align*} Z^{(i)}(X)= F_*(\Xz{i}), \ \  B^{(i)}(X)=F_*(\Xb{i}), \ \ C^{(i)}(X)=F_*(\cone(k_i)), \\
	\lambda^{(i)}\colon=F_* l_i\colon  B^{(i)}(X)\longrightarrow Z^{(i)}(X), i\in \bbZ/N\bbZ,
	\end{align*} 
where $\cone(k_i) $ denotes the cone construction from Definition \ref{mappingcone}. We will now list some additional assumptions that we need in order to assemble the $C^{(i)}$ into a chain complex $C_*$. 

\begin{defn}\label{def:L}
Consider the full subcategory $\cL$ of $\Ho(\cM^{\cC_N})$ consisting of those diagrams $X\in \Ho(\cM^{\cC_N})$ which satisfy the following conditions.
\begin{enumerate}[(i)]
\item The objects $\Xb{i}$ and $\Xz{i}$ are cofibrant in $\cM$ for any $i\in \bbZ/N\bbZ$.
\item The objects $F_*(\Xb{i})$ and $F_*(\Xz{i})$ are contained in $\cB[i]$ for any $i\in \bbZ/N\bbZ$.
\item The map $\lambda^{(i)}\colon F_*(\Xb{i})\to F_*(\Xz{i})$ is a monomorphism for any $i\in \bbZ/N\bbZ$.
\end{enumerate}
\end{defn}
Next we construct a functor 
\[\cQ\colon \cL\lra \twc.\]
Let $X$ be an object of $\cL$. As the functor
\[F_*\colon \Ho(\cM)\lra \cA  \]
is homological, the distinguished triangles
\[ \Xb{i-1} \xrightarrow{k_i} \Xz{i} \lra \cone(k_i)\lra \Sigma\Xb{i-1} \]
induce long exact sequences
\[ \ldots \lra B^{(i-1)}(X)\lra Z^{(i)}(X)\xrightarrow{\iota^{(i)}} C^{(i)}(X) \xrightarrow{\rho^{(i)}} B^{(i-1)}(X)[1]\lra  Z^{(i)}(X)[1] \lra \ldots   .\]
Note that $B^{(i-1)}(X)\in \cB[i-1]$ and $Z^{(i)}(X)\in \cB[i]$ for all $i\in\bbZ/N\bbZ$, since $X\in \cL$. Therefore, the morphisms $B^{(i-1)}(X)\to Z^{(i)}(X)$ and 
$B^{(i-1)}(X)[1]\to  Z^{(i)}(X)[1]$ are zero.
As a consequence, for any $i\in \bbZ/N\bbZ$ we actually obtain short exact sequence in $\cA$
\begin{equation}\label{sexci} 
0\lra  Z^{(i)}(X) \xrightarrow{\iota^{(i)}} C^{(i)}(X) \xrightarrow{\rho^{(i)}} B^{(i-1)}(X)[1]\lra 0.
\end{equation}
Now consider the following objects in $\cA$.
\begin{align*}
C_*(X)&=C^{(0)}(X)\oplus  C^{(1)}(X) \oplus \ldots \oplus C^{(N-1)}(X)\\
Z_*(X)&=Z^{(0)}(X)\oplus  Z^{(1)}(X) \oplus \ldots \oplus Z^{(N-1)}(X)\\
B_*(X)&=B^{(0)}(X)\oplus  B^{(1)}(X) \oplus \ldots \oplus B^{(N-1)}(X)
\end{align*}
The morphisms $\lambda^{(i)}, \iota^{(i)}, \rho^{(i)}, i\in\bbZ/N\bbZ$, induce morphisms between the direct sums

\begin{align*}
\lambda &\colon B_*(X) \lra Z_*(X), & \lambda &= \lambda^{(0)} \oplus \lambda^{(1)} \oplus \ldots \oplus \lambda^{(N-1)}, \\
  \iota &\colon Z_*(X) \lra C_*(X), &   \iota &=   \iota^{(0)} \oplus   \iota^{(1)} \oplus \ldots \oplus   \iota^{(N-1)}, \\
   \rho &\colon C_*(X) \lra B_*(X)[1], & \rho &=    \rho^{(0)} \oplus    \rho^{(1)} \oplus \ldots \oplus \rho^{(N-1)}.
\end{align*}
After summing up, we get a short exact sequence of objects in $\cA$
\begin{equation}\label{sexcstar}
0\lra Z_*(X)\xrightarrow{\iota} C_*(X)\xrightarrow{\rho} B_*(X)[1] \lra 0.
\end{equation}
Splicing this short exact sequence with its shifted copy gives an object in $\twc$. More precisely, define
\[d=\iota[1]\lambda[1]\rho\colon C_*(X)\lra C_*(X)[1].  \]
We have $d^2=0$ by construction and therefore we get a $([1],1)$-twisted complex. We have now arrived at the definition
\[ \cQ\colon \cL\lra \twc, \ \ \ \cQ(X)= \left(\bigoplus_{i\in \bbZ/N\bbZ} F_*(\cone(k_i)),d\right) = (C_*(X),d). \]
It can be shown that $\cQ$ is in fact an equivalence of categories.
The composite
\begin{equation}\label{recfunctor}
\twc \xrightarrow{\cQ^{-1}} \cL\xrightarrow{\hocolim} \Ho(\cM). 
\end{equation}
factors over $\dtwc \lra \Ho(\cM)$, which is Franke's realization functor $\cR.$
 It follows from the construction of $\cR$ that it commutes with suspensions and that $F_*\circ \cR\cong H_*$.

\section{Monoidal Properties of \(\cQ\)}\label{sec:monoidalQ}

In this section, we will examine properties of the bifunctor
\[ i^*\bbL\pr_!(-\wedge^{\bbL}-)\colon \Ho(\cM^{\cC_N}) \times \Ho(\cM^{\cC_N})\lra \Ho(\cM^{\cC_N})\] 
via Theorem \ref{theoremA}, which is one of the main ingredients of the diagram \ref{bigdiagram}.

\subsection{Preliminaries on Crowned Diagrams}

Recall the poset $\cC_N$ from Example \ref{CN} (the crown shape with two rows) and the poset $\cD_N$ from Example \ref{DN} (the crown shape with three rows). We will be interested in two functors between these two categories. The first functor is the \emph{projection functor} 
\begin{align}
\pr\colon \cC_N\times \cC_N&\lra \cD_N  \label{definitionpr} \\
 (\beta_i,\beta_j) &\mapsto \beta_{i+j}  \nonumber\\ 
(\zeta_i,\zeta_j) &\mapsto \zeta_{i+j} \nonumber \\
(\zeta_i,\beta_j) &\mapsto \gamma_{i+j} \nonumber\\
(\beta_i,\zeta_j) &\mapsto \gamma_{i+j} \nonumber. 
\end{align}
Note, that we really should be writing $\beta_{i (\mathrm{mod} N)}$ and $\gamma_{i+j (\mathrm{mod} N)}$ etc. but we commit a small abuse of notation and avoid this. The other functor that we will be 
interested in is the functor 
\begin{equation}\label{definitioni}
i\colon \cC_{N} \lra \cD_{N}, \ \  \zeta_n\mapsto \zeta_n,\ \  \beta_n\mapsto \gamma_n,\end{equation} 
which is the inclusion of the crown shape $\cC_N$ into the bottom two rows of $\cD_N$.
Since weak equivalences in the diagram categories are given objectwise, the functor $i^*\colon \cM^{\cD_N}\lra \cM^{\cC_N}$ 
preserves weak equivalences.
Thus, it defines a functor on the homotopy categories, which we denote by the same letter, that is,
\[i^*\colon \Ho(\cM^{\cD_N})\lra \Ho(\cM^{\cC_N}).  \]
Next, recall the external smash product for diagrams $X\in \cM^{I}$ and $Y\in \cM^{J}$ for $I$ and $J$ direct categories from Definition \ref{objectwisesmash}. By 
choosing $I=J=\cC_N$, it follows formally that we have the bifunctor
\begin{equation}\label{CNobjectwisesmash}
-\wedge-\colon \cM^{\cC_N}\times \cM^{\cC_N}\lra \cM^{\cC_N\times \cC_N}.
\end{equation}
By Proposition \ref{totalderivedobjectwise}, the external product has a total left derived functor
\begin{equation}\label{derivedsmashlone}
-\wedge^{\bbL}-\colon\Ho(\cM^{\cC_N})\times \Ho(\cM^{\cC_N})\lra \Ho(\cM^{\cC_N\times \cC_N}).
\end{equation}
Given diagrams $X,Y\in \Ho(\cM^{\cC_N})$, we can define the homotopy left Kan extension of the external smash product 
$X\wedge^{\bbL}Y \in \Ho(\cM^{\cC_N\times \cC_N})$ along the projection functor $\pr\colon \cC_N\times\cC_N\to \cD_N$, that is, 
\begin{displaymath}
E=\bbL\pr_!(X\wedge^{\bbL}Y) \in \Ho(\cM^{\cD_N}).
\end{displaymath}
Now that we have all the necessary ingredients we can finally state the following theorem.

\begin{thm}\label{theoremA}
The bifunctor
\[i^*\bbL\pr_!(-\wedge^{\bbL}-)\colon \Ho(\cM^{\cC_N}) \times \Ho(\cM^{\cC_N})\lra \Ho(\cM^{\cC_N})\] satisfies the following. Let $X, Y\in \cL$ such that $F_*(\Xa{n}), F_*(\Ya{n})\in \cAp$ for any $n\in \bbZ/N\bbZ$ and any $\alpha\in \left\{\beta,\zeta\right\}$.
Then, $i^*\bbL\pr_!(X\wedge^{\bbL}Y) \in \cL$, that is to say, we have a bifunctor
\[i^*\bbL\pr_!(-\wedge^{\bbL}-)\colon \cL\x \cL\lra \cL.\]
Furthermore, there is a natural isomorphism
\begin{equation*}
\cQ(i^*\bbL\pr_!(X\wedge^{\bbL}Y))\cong \cQ(X)\otimes \cQ(Y).
\end{equation*}
\end{thm}

The theorem has two parts. 
First, we show that  
 $i^*\bbL\pr_!(-\wedge^{\bbL}-)$ is in fact a bifunctor
\[i^*\bbL\pr_!(-\wedge^{\bbL}-)\colon \cL\x \cL\lra \cL.\] 
The second part is that for any two crowned diagrams $X,Y \in \cL$ 
satisfying the stated hypotheses, there is a natural isomorphism
\begin{equation*}
\cQ(i^*\bbL\pr_!(X\wedge^{\bbL}Y))\cong \cQ(X)\otimes \cQ(Y).
\end{equation*}
The two parts combined us that the following diagram commutes (up to natural isomorphism)
\[
\xymatrix{
 \twc    \times  \twc  \ar[d]_{\otimes} & \ar[l]_{\ \ \ \ \ \ \ \ \ \ \ \ \cQ\x\cQ} \cL\times \cL \ar[d]^{i^*\bbL\pr_!} \\
 \twc                 & \cL \ar[l]^{\cQ}.}
\]
The first part of Theorem \ref{theoremA} is the content of Subsection \ref{spectralsequence} and Proposition \ref{propA}. The natural isomorphism 
\begin{equation*}
\cQ(i^*\bbL\pr_!(X\wedge^{\bbL}Y))\cong \cQ(X)\otimes \cQ(Y).
\end{equation*}
is the content of Subsections \ref{newcones} and \ref{differentials} and Proposition \ref{propB}.

\subsection{Slice categories of the projection functor}\label{sec:slicecats}
Again, the values of $\bbL\pr_!(-\wedge^{\bbL}-)$ are given by the formula in Proposition \ref{holanformula}. That is, the values of $E$ at the objects of $\cD_N$ are given by
\begin{align}
E_{\gamma_n}&=\hocolim_{\pr/\gamma_n}(X\wedge^{\bbL} Y) \label{value1} \\
E_{\zeta_n}&=\hocolim_{\pr/\zeta_n}(X\wedge^{\bbL} Y) \label{value2} \\ 
E_{\beta_n}&=\hocolim_{\pr/\beta_n}(X\wedge^{\bbL} Y). \label{value3} 
\end{align}
The structure morphisms of the diagram $E$,  $\widehat{l}_n\colon E_{\gamma_n}\to E_{\zeta_n}$ and $\widehat{k}_n\colon E_{\gamma_{n+1}}\to E_{\zeta_n}$, are the edges of the homotopy Kan extension and are given by the natural maps
\begin{align}
E_{\gamma_n}\cong \hocolim_{\pr/\gamma_n}(X\wedge^{\bbL}Y)&\lra \hocolim_{\pr/\zeta_n}(X\wedge^{\bbL} Y)\cong E_{\zeta_n} \label{gammanzetan}   \\
E_{\gamma_{n+1}}\cong \hocolim_{\pr/\gamma_{n+1}}(X\wedge^{\bbL}Y)&\lra \hocolim_{\pr/\zeta_n}(X\wedge^{\bbL} Y)\cong E_{\zeta_n} \label{gammanonezetan}
\end{align}
induced by the maps of posets $\phi$ and $\psi$, respectively, see \ref{psiphi}.

Since we are interested in the homotopy Kan extension of the functor $\pr\colon \cC_N\times \cC_N\lra \cD_N$, we need to identify all the slice categories involved, \ie, 
$\pr/\zeta_n, \pr/\gamma_n$ and $\pr/\beta_n$. We have the following three cases.
\begin{enumerate}[(i)]
\item \label{zetan} The first case is $\pr/\zeta_n$. For $n\in \bbZ/N\bbZ$ and the object $\zeta_n$ we have the slice category $\pr/\zeta_n$
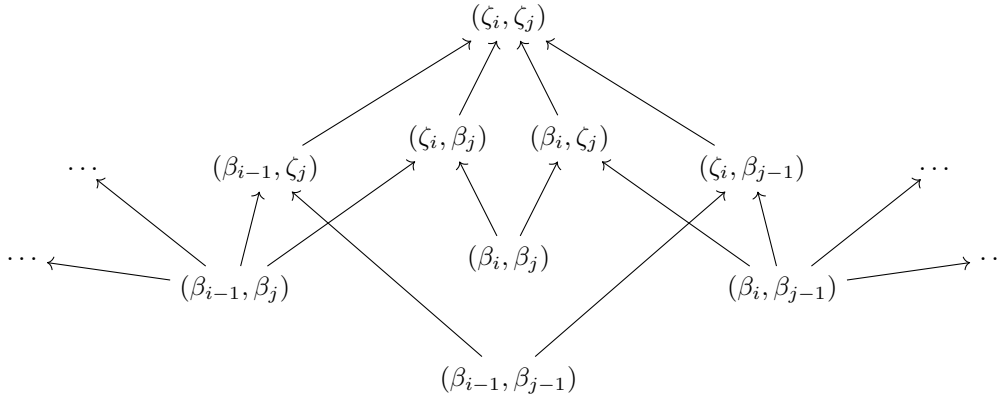
\begin{equation*}
\begin{tikzpicture}[scale=.8]
  \node (one) at (0,2) {$  (\zeta_i,\zeta_j)$};     
  \node (a) at (-4,-.5) {$(\beta_{i-1},\zeta_j)$};
  \node (b) at (-1,0) {$(\zeta_i,\beta_j)$};
  \node (c) at (1,0) {$(\beta_i,\zeta_j)$};
  \node (d) at (4,-0.5) {$(\zeta_i,\beta_{j-1})$};
  \node (zero) at (0,-2) {$(\beta_i,\beta_j)$}; 
	\node (minus) at (0,-4) {$(\beta_{i-1},\beta_{j-1})$};
	\node (minusone) at  (-4.5,-2.5) {$(\beta_{i-1},\beta_j)$};
	\node (minustwo) at  (4.5,-2.5) {$(\beta_i,\beta_{j-1})$}; 
	\node (extraleftone) at (-7,-.5)  {$\ldots$};   
	\node (extralefttwo) at (-8,-2) {$\ldots$};
	\node (extrarightone) at (8,-2) {$\ldots$};
	\node (extrarighttwo) at (7,-.5) {$\ldots$};     
	\draw  [->]   (zero)--(b) ; 
	\draw [->]  (zero)--(c) ; 
\draw	 [->](b)-- (one) ;
\draw	 [->](c)--(one) ;
	\draw [->]  (minus)--(a); 
	\draw [->]   (minus)-- (d) ; 
\draw	 [->]  (minusone)--(a) ; 
	\draw [->]  (minusone) --(b) ; 
\draw	 [->]  (minustwo)--(c) ; 
	\draw [->] (minustwo)--(d); 
\draw  [->]	(a)--(one) node[midway,above left] {$$}   ;
\draw	 [->] (d)--(one) node[midway, above right] {$$} ;
\draw [->]  (minusone) --(extraleftone) ;
\draw [->] (minusone)--(extralefttwo) ;
\draw [->] (minustwo)-- (extrarightone);
\draw [->] (minustwo)-- (extrarighttwo);
\end{tikzpicture}
\end{equation*}
where $i+j\equiv n \ (\mathrm{mod}\ N)$. Note that all the non-identity morphisms are of the form $(1,l_i)$ or $(l_i,1)$ and similarly $(1,k_i)$ or $(k_i,1)$ for any $i\in\bbZ/N\bbZ$. The poset 
$\pr/\zeta_n$ follows the same pattern to the left and to the right.

\item \label{gamman}
Next we have the case $\pr/\gamma_n$. Let $n\in \bbZ/N\bbZ$ and consider now the slice category $\pr/\gamma_n$ which looks as follows,
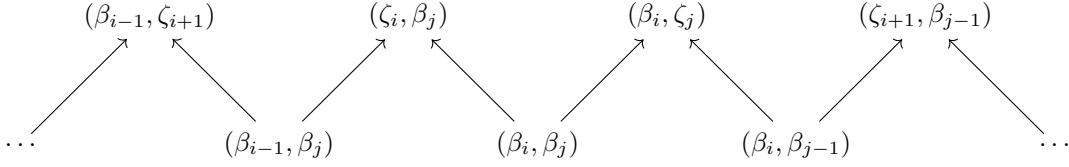
\begin{equation*}
\begin{tikzpicture}[scale=1.7]
\node (zero)at (0,0) {$(\beta_i,\beta_j)$};
\node (one) at (1,1) {$(\beta_i,\zeta_j)$};
\node (two) at (-1,1){$(\zeta_i,\beta_j)$};
\node (three) at (2,0){$(\beta_i,\beta_{j-1})$};
\node (four)  at (-2,0){$(\beta_{i-1},\beta_j)$};
\node (five) at  (3,1){$(\zeta_{i+1},\beta_{j-1})$} ;
\node (six) at   (-3,1) {$(\beta_{i-1},\zeta_{i+1})$};
\node (seven) at  (4,0) {$\ldots$};
\node (eight) at (-4,0) {$\ldots$};
\draw [->] (zero)-- (one) ;
\draw [->] (zero)-- (two);
\draw [->](three)-- (one);
\draw[->] (four)-- (two);
\draw [->](three)-- (five);
\draw [->](four)--(six);
\draw [->](seven)-- (five);
\draw [->](eight)--(six);
\end{tikzpicture}
\end{equation*}
where $i+j \equiv n\ (\mathrm{mod}\ N)$. Similarly to the above all the non-identity morphisms are of the form $(1,l_i)$ or $(l_i,1)$ and $(1,k_i)$ or $(k_i,1)$ for any
$i\in \bbZ/N\bbZ$.

\item \label{betan}
Next is the case $\pr/\beta_n$. Let again $n\in \bbZ/N\bbZ$ but now we consider the slice category $\pr/\beta_n$. Notice that it is
\begin{equation*}
\begin{tikzcd}
\ldots &  & (\beta_{i-1},\beta_{j+1}) &  & (\beta_i,\beta_j) &  & (\beta_{i+1},\beta_{j-1}) &  & \ldots
\end{tikzcd}
\end{equation*}
in which $i+j \equiv n\ (\mathrm{mod}\ N)$. In other words, it is a discrete category. This means that 
\[E_{\beta_n}=\hocolim_{\pr/\beta_n}(X\wedge^{\bbL}Y)\cong \bigoplus_{i+j=n}\Xb{i}\wedge^{\bbL}\Yb{j}.   \]
This is the only case that we can be explicit about the values of the homotopy left Kan extension $E=\bbL\pr_!(X\wedge^{\bbL}Y)$.

\item \label{subposetzetan} 
Our last example is a particular subposet of $\pr/\zeta_n$ and it is not strictly speaking a slice of any value. However it will be very useful for us is the following. Consider the following subposet $J_n\subseteq \pro/\zeta_n$ defined as follows
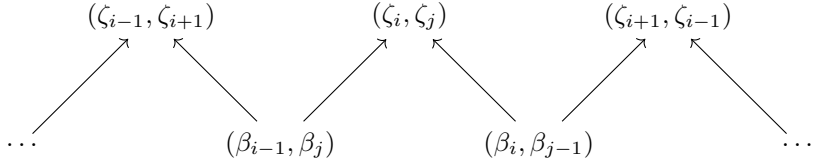
\begin{equation*}
\begin{tikzpicture}[scale=1.7]
\node (zero) at (0,1) {$(\zeta_i,\zeta_j)$};
\node (one) at (1,0) {$(\beta_i,\beta_{j-1})$};
\node (two) at (2,1) {$(\zeta_{i+1},\zeta_{i-1})$};
\node (three)at ( 3,0) {$\ldots$};
\node (four) at (-1,0) {$(\beta_{i-1},\beta_j)$};
\node (five) at (-2,1) {$(\zeta_{i-1},\zeta_{i+1})$};
\node (six) at (-3,0) {$\ldots$};N
\draw [->] (one)-- (zero);
\draw [->] (one)-- (two);
\draw [->] (three)-- (two);
\draw [->] (four)-- (zero);
\draw [->] (four)-- (five);
\draw [->] (six)-- (five);
\end{tikzpicture}
\end{equation*}
where $i+j\equiv n\  (\mathrm{mod}\ N)$. In this poset, the non-identity morphisms are of the form $(k_i,l_i)$ or $(l_i,k_i)$, unlike the examples above where one arrow was always the identity arrow. 
\end{enumerate}

\begin{remark}\label{thetanisrightadjoint}
Now let $\theta\colon J_n \to \pr/\zeta_n$ denote the inclusion of the subposet defined in (\ref{subposetzetan}) into the poset in (\ref{zetan}). We will define a map of posets
\begin{equation*}
L\colon \pr/\zeta_n\lra J_n
\end{equation*}
where it suffices to define it for the part of the poset visible in (\ref{zetan}) as the rest can be defined analogously. The map $L$ is given by
\begin{align*}
L\colon \pr/\zeta_n  &\lra J_n  \\ 
(\beta_{i+1},\beta_j) &\mapsto (\beta_{i+1},\beta_j) \\
(\beta_i,\beta_{j+1}) &\mapsto (\beta_i,\beta_{j+1}) \\ 
  \mbox{else} \ \ &\mapsto (\zeta_i,\zeta_j).
 \end{align*}
We note that $L$ left adjoint to $\theta$- this can quickly be verified straight from the definition as the morphism sets in either poset are either empty or consist of exactly one element.
As a consequence, since the inclusion map $\theta\colon J_n\lra \pro/\zeta_n$ is a right adjoint, it is homotopy final, \ie, for any 
$F\in \Ho(\cM^{\pr/\zeta_n})$ we have
\[\hocolim_{J_n}\theta^*(F)\cong \hocolim_{\pr/\zeta_n}F  .\]
In other words, the value $E_{\zeta_n}$ in (\ref{value1}) can be calculated as
\begin{equation}\label{Ezetansecond}
\Ez{n}\cong\hocolim_{\pr/\zeta_n}(X\wedge^{\bbL} Y) \cong \hocolim_{J_n}\theta^*(X\wedge^{\bbL} Y).
\end{equation}
We discuss homotopy finality in more detail in Section \ref{sec:hocolimcalcs}, see Definition \ref{def:homotopyfinal}.
\end{remark}

Given any of subposet of $\cC_N\times \cC_N$, \eg, $\prog{n}$ from Example (\ref{gamman}), we can define the restriction of the external smash product $X\wedge Y\in \cM^{\cC_N\times \cC_N} $ to $\prog{n}$ by 
taking the pullback along the inclusion $\nu\colon \prog{n}\lra \cC_N\times \cC_N $, that is, 
\[\nu^*\colon \cM^{\cC_N\times \cC_N}\lra \cM^{\prog{n}}.\]
Notice that $\nu^*$ preserves weak equivalences so it induces a functor on homotopy categories
\[ \nu^*\colon \Ho(\cM^{\cC_N\times \cC_N})\lra \Ho(\cM^{\prog{n}}).\]
Moreover, we have maps between the subposets of $\cC_N\times \cC_N$. The morphisms $\gamma_n\rightarrow \zeta_n$ and $\gamma_{n+1}\rightarrow \zeta_n$ induce maps of posets
\begin{align}\label{psiphi}
\psi\colon\pr/\gamma_n &\lra \pr/\zeta_n  \\
\phi\colon\pr/\gamma_{n-1}&\lra \pr/\zeta_n \nonumber
\end{align}
which in turn also induce pullback functors on the homotopy categories, that is,
\[\phi^*\colon \Ho( \cM^{\proz{n}})\lra \Ho( \cM^{\prog{n-1}}), \ \ \text{and}\ \ \psi^*\colon \Ho( \cM^{\proz{n}})\lra \Ho( \cM^{\prog{n}}).    \]
We conclude this section with a convention.
\begin{conv}
Because of the above, we will commit an abuse of notation and instead of writing, for example,
\begin{equation*}
\phi^*(X\wedge^{\bbL}Y)\in \Ho(\cM^{\pr/\gamma_n})
\end{equation*} 
we will simply write 
\begin{equation*}
X\wedge^{\bbL}Y \in \Ho(\cM^{\pr/\gamma_n}),
\end{equation*} 
with the understanding that this diagram was given by a composition of restriction functors
\[ \Ho(\cM^{\cC_N\times \cC_N}) \xrightarrow{\pi^*}  \Ho(\cM^{\pr/\zeta_n}) \xrightarrow{\phi^*}  \Ho(\cM^{\pr/\gamma_n})   \]
unless we need the extra notation for clarification. 
\end{conv}

\begin{remark}
Consider a diagram $F\in \Ho (\cM^{\cC_N\times \cC_N})$. By Convention \ref{homotopycategoryofM}, we can assume that $F$ is a projective cofibrant object, so in particular, it is objectwise cofibrant. The 
external smash product 
\[-\wedge -\colon  \cM^{\cC_N}\times  \cM^{\cC_N}\lra  \cM^{\cC_N\times \cC_N} \]
as defined in \eqref{CNobjectwisesmash}
is Quillen bifunctor, so in particular it preserves cofibrant objects. 
This implies that $X\wedge^{\bbL} Y$ is cofibrant in $\cM^{\cC_N\times \cC_N}$, so in particular it is objectwise cofibrant. Now, for any subposet 
$$\iota\colon J\hookrightarrow \cC_N\times \cC_N,$$ \eg, any of the slice categories of the projection functor $\pr$ \eqref{definitionpr}, we have the pullback functor
\[\iota^*\colon \cM^{\cC_N\times \cC_N} \lra \cM^{J}.\]
This functor is not necessarily a left Quillen functor with respect the projective model structures, see Proposition \ref{modelondiagrams}. However, the diagram $\iota^*(X\wedge^{\bbL} Y)$, is objectwise cofibrant, 
which means that the geometric realization of the simplicial replacement still models the homotopy colimit of the diagram $\iota^*(X\wedge^{\bbL} Y)$. In particular, the skeletal filtration of all the restrictions 
is always Reedy cofibrant, see Lemma \ref{objectwisecofibrantreedy}.
\end{remark}

\subsection{Spectral Sequence Calculations}\label{spectralsequence}
The main result of this subsection is that given crowned diagrams $X,Y\in \cL$ that for
satisfying a simple condition, the diagram $i^*E=i^*\bbL\pr_!(X\wedge^{\bbL} Y)$ is also in the subcategory $\cL$, i.e. the objects $\Xb{i}$ and $\Xz{i}$ are cofibrant in $\cM$,
the objects $F_*(\Xb{i})$ and $F_*(\Xz{i})$ are in $\cB[i]$, and
the map $$\lambda^{(i)}\colon F_*(\Xb{i})\to F_*(\Xz{i})$$ is a monomorphism for any $i\in \bbZ/N\bbZ$, see Definition \ref{def:L}.
Essentially, this condition is that for the given homological functor $F_*\colon \Ho(\cM)\to \cA$, either the crowned diagram $X$ or $Y$ is objectwise projective.

\begin{prop}\label{propA}
Let $X, Y \in \cL$ such that $F_*(X_{{\alpha}_n}), F_*(Y_{{\alpha}_n})\in  \cAp$ for any $n\in \bbZ/N\bbZ$ and any $\alpha\in \left\{\beta,\zeta\right\}$. Consider the homotopy left Kan 
extension $E=\bbL\pr_!(X\wedge^{\bbL}Y)\in \Ho(\cM^{\cD_N})$ of 
\begin{equation*}
X\wedge^{\bbL}Y\in \Ho(\cM^{\cC_N\times \cC_N})
\end{equation*} along 
\begin{displaymath}
\pro\colon \cC_N\times \cC_N\lra \cD_N
\end{displaymath}
with the values and morphisms given in (\ref{value1})-(\ref{value3}) and (\ref{gammanzetan}), (\ref{gammanonezetan}) respectively.
\begin{equation*}
\xymatrix{E_{\zeta_0}  &  E_{\zeta_1}  & \ldots & E_{\zeta_{N-1}}   \\ 
E_{\gamma_0} \ar[u] \ar[ur] & \ar[u]  \ar[ur] E_{\gamma_1}   & \ar[ur] \ldots & \ar[u] \ar[ulll] E_{\gamma_{N-1}}\\
   E_{\beta_0} \ar[u] \ar[ur] & \ar[u]  \ar[ur] E_{\beta_1}   & \ar[ur] \ldots & \ar[u] \ar[ulll] E_{\beta_{N-1}}    }
\end{equation*}
Then, for any $n\in \bbZ/N\bbZ$ and any $\alpha\in \left\{\beta,\zeta\right\}$ we have  $F_*(E_{\alpha_n})\in \cB[n]$ and the morphisms 
\begin{displaymath}
F_*(E_{\gamma_n})\lra F_*(E_{\zeta_n})
\end{displaymath}
 induced by $E_{\gamma_n}\lra E_{\zeta_n}$ are monomorphisms.
\end{prop}

\begin{cor}\label{pullbackviaiisine}
Let $X,Y$ be crowned diagrams satisfying the hypothesis of Proposition \ref{propA}. The top two rows of the diagram $E=\bbL\pr_!(X\wedge^{\bbL}Y)$ form an object in $\cL$, that is, the diagram $i^*E\in \cL$.
\end{cor}

By our assumption, for any $n\in\bbZ/N\bbZ$ and any $\alpha\in \left\{\beta,\zeta\right\}$ the objects $F_*(X_{{\alpha}_n})$ and  $F_*(Y_{{\alpha}_n})$ are projective in $\cA$. This ensures that 
there are natural K{\"u}nneth isomorphisms
\begin{equation}\label{Kunneth}
  	F_*\left(X_{{\alpha}_n}\wedge^{\bbL} Y_{{\alpha}_n}  \right) \cong F_*( X_{{\alpha}_n} ) \otimes F_*(Y_{{\alpha}_n}).
\end{equation}  
Since the values $E_{\zeta_n}, E_{\gamma_n}$ and $E_{\beta_n}$ are computed via homotopy colimits, we will use \eqref{sshomotopycolimit}, the spectral sequences converging to the 
homology of the homotopy colimit.
\begin{lemma}\label{sseqlemma}
There are spectral sequences
\begin{equation}\label{ssgamman} 
E^2_{pq}= H_p(\pr/\gamma_n; F_q(X\wedge^{\bbL}Y))\Rightarrow F_{p+q}\left(\hocolim_{\pr/\gamma_n}(X\wedge^{\bbL}Y)\right)\cong F_{p+q}(E_{\gamma_n})
\end{equation}
and
\begin{equation}\label{sszetan}
E^{\prime 2}_{pq}= H_p(\pr/\zeta_n; F_q(X\wedge^{\bbL}Y))\Rightarrow F_{p+q}\left(\hocolim_{\pr/\zeta_n}(X\wedge^{\bbL}Y) \right)\cong F_{p+q}(E_{\zeta_n})
\end{equation}
and natural morphisms of spectral sequences $f\colon\left\{E^2_{pq}\right\}\lra \left\{E^{\prime 2}_{pq}\right\}$ induced by the map in (\ref{gammanzetan}).
\end{lemma}
We will now begin the proof of Proposition \ref{propA}.

\begin{proof}
Our claim is that $F_*(E_{\gamma_n})\lra F_*(E_{\zeta_n})$ is a monomorphism, where $$E=\bbL\pr_!(X\wedge^{\bbL}Y)\in \Ho(\cM^{\cD_N})$$ as before and $F_*\colon \Ho(\cM)\to \cA$ is our homological functor. 
To obtain information on $F_*(E_{\gamma_n})$ and  $F_*(E_{\zeta_n})$, we will start by working out the spectral sequence (\ref{ssgamman}), which we explained is a special case of the spectral sequence (\ref{sshomotopycolimit}). 
The proof of the proposition is divided into three parts:
\begin{itemize}
\item calculating the $E_2$-term $H_p(\pr/\gamma_n; F_q(X\wedge^{\bbL}Y))$,
\item calculating the $E_2$-term $H_p(\pr/\zeta_n; F_q(X\wedge^{\bbL}Y))$,
\item showing that the induced map of spectral sequences gives the desired isomorphism.
\end{itemize}

\bigskip
\noindent\fbox{%
    \parbox{15em}{%
       {\bf Step 1:} $H_p(\pr/\gamma_n; F_q(X\wedge^{\bbL}Y))$
    }%
}

\bigskip
We will use the simplicial replacement techniques explained in Section \ref{sec:homologyofacategory}.
Recall the poset $\pr/\gamma_n$ which looks as follows.
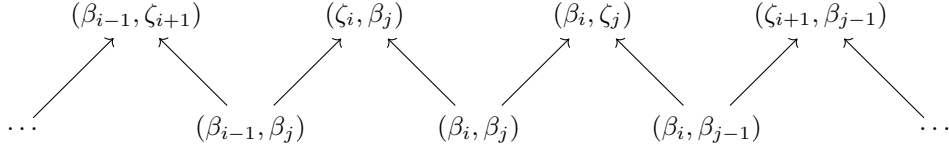
\begin{equation*}
\begin{tikzpicture}[scale=1.5]
\node (zero)at (0,0) {$(\beta_i,\beta_j)$};
\node (one) at (1,1) {$(\beta_i,\zeta_j)$};
\node (two) at (-1,1){$(\zeta_i,\beta_j)$};
\node (three) at (2,0){$(\beta_i,\beta_{j-1})$};
\node (four)  at (-2,0){$(\beta_{i-1},\beta_j)$};
\node (five) at  (3,1){$(\zeta_{i+1},\beta_{j-1})$} ;
\node (six) at   (-3,1) {$(\beta_{i-1},\zeta_{i+1})$};
\node (seven) at  (4,0) {$\ldots$};
\node (eight) at (-4,0) {$\ldots$};
\draw [->] (zero)-- (one) ;
\draw [->] (zero)-- (two);
\draw [->](three)-- (one);
\draw[->] (four)-- (two);
\draw [->](three)-- (five);
\draw [->](four)--(six);
\draw [->](seven)-- (five);
\draw [->](eight)--(six);
\end{tikzpicture}
\end{equation*}
Here, $i+j=n  \ \mathrm{mod}\ N$, and so the functor $X\wedge^{\bbL}Y \in \Ho(\cM^{\pr/\gamma_n}) $ looks as follows.
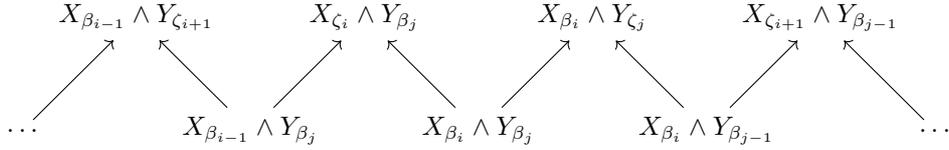
\begin{equation*}
\begin{tikzpicture}[scale=1.5]
\node (zero)at (0,0) {$X_{\beta_i}\wedge Y_{\beta_j}$};
\node (one) at (1,1) {$X_{\beta_i}\wedge Y_{\zeta_j}$};
\node (two) at (-1,1){$X_{\zeta_i}\wedge Y_{\beta_j}$};
\node (three) at (2,0){$X_{\beta_i}\wedge Y_{\beta_{j-1}}$};
\node (four)  at (-2,0){$X_{\beta_{i-1}}\wedge Y_{\beta_j}$};
\node (five) at  (3,1){$X_{\zeta_{i+1}}\wedge Y_{\beta_{j-1}}$} ;
\node (six) at   (-3,1) {$X_{\beta_{i-1}}\wedge Y_{\zeta_{i+1}}$};
\node (seven) at  (4,0) {$\ldots$};
\node (eight) at (-4,0) {$\ldots$};
\draw [->](zero)-- (one) ;
\draw [->](zero)-- (two);
\draw [->](three)-- (one);
\draw [->] (four)-- (two);
\draw [->] (three)-- (five);
\draw [->](four)--(six);
\draw [->] (seven)--(five);
\draw [->] (eight)--(six);
\end{tikzpicture}
\end{equation*}
Our goal is to compute 
\begin{displaymath}
H_p(\pr/\gamma_n;F_q(X\wedge^{\bbL} Y))\ \ \ \  \mbox{ for all $p\geq 0$ and all $q\in\bbZ$},
\end{displaymath}
which are the $E^2$-terms of the spectral sequence (\ref{ssgamman}). In order to do so, we apply 
the homological functor $F_n(-)$ to the previous diagram to get the diagram $F_n(X\wedge^{\bbL} Y)\in \cA^{\pr/\gamma_n}$ which, by \eqref{Kunneth} is the following.
 \begin{equation}\label{engamman} 
\begin{tikzpicture}[scale=1.5][baseline= (current  bounding  box.north)]               
\node (zero)at (0,0) {$B^{(i)}\otimes \widetilde{B}^{(j)}$};
\node (one) at (1,1) {$B^{(i)}\otimes \widetilde{Z}^{(j)}$};
\node (two) at (-1,1){$Z^{(i)}\otimes \widetilde{B}^{(j)}$};
\node (three) at (2,0){$0$};
\node (four)  at (-2,0){$0$};
\node (five) at  (3,1){$B^{(i+1)}\otimes \widetilde{Z}^{(j-1)}$} ;
\node (six) at   (-3,1) {$B^{(i-1)}\otimes \widetilde{Z}^{(j+1)}$};
\node (seven) at  (4,0) {$\ldots$};
\node (eight) at (-4,0) {$\ldots$};
\draw [->](zero)-- (one) ;
\draw [->] (zero)-- (two);
\draw [->] (three)-- (one);
\draw [->] (four)-- (two);
\draw [->] (three)-- (five);
\draw [->] (four)--(six);
\draw [<-] (six)--(eight);
\draw  [<-](five)--(seven);
\end{tikzpicture}
\end{equation}
We will write 
\begin{align}
f_{ij}=\lambda_i\otimes 1  \colon B^{(i)}\otimes\widetilde{B}^{(j)}&\lra Z^{(i)}\otimes\widetilde{B}^{(j)} \label{fij}\\ 
g_{ij}=1\otimes \tilde{\lambda}_j    \colon B^{(i)}\otimes\widetilde{B}^{(j)}&\lra B^{(i)}\otimes \widetilde{Z}^{(j)} \label{gij}
\end{align} 
to distinguish, for labeling purposes, the two different morphisms in the simplicial replacement below. Note that since $B^i$ and $\widetilde{B}^j$ and projective in $\cA$, by our convention they are 
automatically flat, hence the morphisms \eqref{fij} and \eqref{gij} are monomorphisms. 

Next, we consider the simplicial replacement of the diagram  $F_n(X\wedge^{\bbL}Y)\in \cA^{\pr/\gamma_n}$, that 
is $\srep\left(F_n(X\wedge^{\bbL}Y)\right)\in \cA^{\Dopp}$. Following Definition \ref{simprep} we have that
\begin{align*} 
\displaystyle{\srep\left(F(X\wedge^LY)\right)_0}&= \bigoplus_{i+j=n}\left( (B^{(i)}\otimes\widetilde{B}^{(j)}) \oplus (Z^{(i)}\otimes\widetilde{B}^{(j)})\oplus (B^{(i)}\otimes \widetilde{Z}^{(j)}) \right)\\
\srep\left(F_n(X\wedge^{\bbL}Y)\right)_1&= \bigoplus_{i+j=n} \left( (B^{(i)}\otimes\widetilde{B}^{(j)})_{f_{ij}} \oplus  (B^{(i)}\otimes\widetilde{B}^{(j)})_{g_{ij}}\right),
\end{align*}
with face maps given by ``source'' and ``target'' respectively.  Because of the shape of the poset $\prog{n}$, for all $m \geq 2$ the simplices $\srep\left(F_n(X\wedge^{\bbL}Y)\right)_m$ 
consist solely of degenerate simplices. 

Now we consider the associated complex $C_*(F_n(X\wedge^{\bbL}Y))$
of this simplicial complex, see Definition \ref{associatedcomplex}. We briefly explain the differential of the complex $C_*(E(1)_{-n}(X\wedge^{\bbL}Y))$, namely the 
map
\[d=d_0-d_1\colon C_1(F_n(X\wedge^{\bbL}Y))\lra C_0(F_n(X\wedge^{\bbL}Y)) .\]
Notice from \eqref{engamman}, we can consider the simpler case where the diagram looks as follows.
\[
\begin{tikzcd}
 Z^{(i)}\otimes \tB^{(j)} &                  &  B^{(i)}\ox \tZ^{(j)} \\
                  &B^{(i)}\otimes \tB^{(j)}  \arrow[ur,"g_{ij}", swap] \arrow[ul,"f_{ij}"] &
\end{tikzcd}
\] 
Then, the differential of the associated complex of the simplicial replacement of this diagram is
\begin{align*} 
d_{ij}=d_0-d_1\colon (B^{(i)}\otimes\widetilde{B}^{(j)}) \oplus  (B^{(i)}\otimes\widetilde{B}^{(j)}) &\lra (B^{(i)}\otimes\widetilde{B}^{(j)})\oplus  (Z^{(i)}\otimes\widetilde{B}^{(j)})\oplus (B^{(i)}\otimes \widetilde{Z}^{(j)})\\
 (x,y)  &\mapsto (x+y, -f_{ij}(x), -g_{ij}(y)).
\end{align*}
The $0$th homology of the complex is just the pushout $$ B^{(i)}\ox \tZ^{(j)} \coprod_{B^{(i)}\otimes \tB^{(j)} } Z^{(i)}\otimes \tB^{(i)} .$$ The first homology is the kernel of the differential $d_{ij}$. Since the maps
$f_{ij}$ and $g_{ij}$ are injective, this forces $d_{ij}(x,y)=0$ if and only if $x=y=0$, which implies that the first homology is trivial. 
It follows from the diagram \eqref{engamman} that the differential $d$ on the complex $C_*(F_n(X\wedge^{\bbL}Y))$ is the direct sum of the 
differentials $d_{ij}$ for $i+j=n$. Now that we know the differential of the complex $C_*(F_n(X\wedge^{\bbL}Y))$ we will compute its homology. It follows that 
$H_0(\pr/\gamma_n;F_n(X\wedge^\bbL Y))$ is the 
colimit of the diagram $F_n(X\wedge^{\bbL} Y)$. By inspecting the diagram $F_n(X\wedge^{\bbL} Y)$ above we can see 
the colimit of the diagram is a direct sum (coproduct) of 
pushouts, that is,
\begin{align*}
H_0(\pr/\gamma_n;F_n(X\wedge^{\bbL} Y)) &=\colim_{\pr/\gamma_n}F_n(X\wedge^{\bbL} Y) \\
&=\bigoplus_{i+j=n}\left( Z^{(i)}\otimes \widetilde{B}^{(j)}\coprod_{B^{(i)}\otimes \widetilde{B}^{(j)}}B^{(i)}\otimes \widetilde{Z}^{(j)}\right).
\end{align*}
Similar to the simpler case, the first homology
\begin{displaymath}
H_1(\pr/\gamma_n; F_n(X\wedge^{\bbL}Y))
\end{displaymath} is the kernel of the differential
\begin{displaymath}
d_0-d_1\colon\srep(F_n(X\wedge^{\bbL}Y))_1\lra \srep(F_n(X\wedge^{\bbL}Y))_0.
\end{displaymath} 
Since it is a direct sum of the simpler differentials $d_{ij}$ as above, it follows that
\begin{displaymath}
H_1(\pr/\gamma_n;F_n(X\wedge^{\bbL}Y))=0.
\end{displaymath}
Of course, all the higher homologies 
\begin{displaymath}
H_q(\pr/\gamma_n;F_n(X\wedge^{\bbL}Y))
\end{displaymath}
vanish for all $q\geq 2$.

\vspace{3mm}
Next, we apply the homological functor $F_{n-1}(-)$ to the diagram $X\wedge^{\bbL}Y\in\Ho(\cM^{\pr/\gamma_n})$ and we have the diagram 
$F_{n-1}(X\wedge^{\bbL}Y)\in \cA^{\pr/\gamma_n}$ which looks as follows.
\begin{equation*}
\begin{tikzpicture}
\node (zero)at (0,0) {$0$};
\node (one) at (1,1) {$0$};
\node (two) at (-1,1){$0$};
\node (three) at (2,0){$B^{i}\otimes \widetilde{B}^{(j-1)}$};
\node (four)  at (-2,0){$B^{(i-1)}\otimes \widetilde{B}^{(j)} $};
\node (five) at  (3,1){$0$} ;
\node (six) at   (-3,1) {$0$};
\node (seven) at  (4,0) {$\ldots$};
\node (eight) at (-4,0) {$\ldots$};
\draw [->](zero)-- (one) ;
\draw [->](zero)-- (two);
\draw [->](three)-- (one);
\draw [->](four)-- (two);
\draw [->](three)-- (five);
\draw [->](four)--(six);
\draw [<-] (six)--(eight);
\draw [<-](five)--(seven);
\end{tikzpicture}
\end{equation*}
Clearly,
\[H_0(\pr/\gamma_n;F_{n-1}(X\wedge^{\bbL}Y))=0, \]
and
\[H_1(\pr/\gamma_n;F_{n-1}(X\wedge^{\bbL}Y))=\bigoplus_{i+j=n+1}B^{(i)}\otimes \widetilde{B}^{(j)} .\]

It follows that for all $p\geq 0$ and all $m\neq -n, -n-1\ \ \mathrm{mod}\ N$, the terms $H_p(\pr/\gamma_n;F_m(X\wedge^{\bbL}Y))$ all vanish. This completes the computation of
the $E^2$ terms of the spectral sequence. It is concentrated in degrees $(0,m)$ and $(1,m-1)$ with $m\equiv n \ \operatorname{mod} N$. Therefore the spectral sequence collapses and we have a short exact sequence
\begin{equation*}
0\lra \bigoplus_{i+j=n}\left( Z^{(i)}\otimes \widetilde{B}^{(j)}\bigoplus_{B^{(i)}\otimes \widetilde{B}^{(j)}}B^{(i)}\otimes \widetilde{Z}^{(j)}\right) \lra F_n(E_{\gamma_n}) \lra 
\bigoplus_{i+j=n+1}B^{(i)}\otimes \widetilde{B}^{(j)}\lra 0.
\end{equation*}
This concludes the calculation of the spectral sequence \eqref{ssgamman}.

\bigskip
\noindent\fbox{%
    \parbox{15em}{%
       {\bf Step 2:} $H_p(\pr/\zeta_n; F_q(X\wedge^{\bbL}Y))$
    }%
}

\bigskip
We will now repeat the previous strategy and apply it to the spectral sequence \eqref{sszetan}. Recall the poset $J_n$ from (\ref{subposetzetan}), which is the following subposet of $\pr/\zeta_n$.
\begin{equation*}
\begin{tikzpicture}[scale=1.5]
\node (zero) at (0,1) {$(\zeta_i,\zeta_j)$};
\node (one) at (1,0) {$(b_i,b_{j-1})$};
\node (two) at (2,1) {$(\zeta_{i+1},\zeta_{i-1})$};
\node (three)at ( 3,0) {$\ldots$};
\node (four) at (-1,0) {$(\beta_{i-1},\beta_j)$};
\node (five) at (-2,1) {$(\zeta_{i-1},\zeta_{i+1})$};
\node (six) at (-3,0) {$\ldots$};
\draw [->] (one)-- (zero);
\draw [->] (one)-- (two);
\draw [->] (three)-- (two);
\draw [->] (four)-- (zero);
\draw [->] (four)-- (five);
\draw [->] (six)-- (five);
\end{tikzpicture}
\end{equation*}
By Remark \ref{thetanisrightadjoint}, the inclusion functor $\theta\colon J_n\to \pr/\zeta_n$ has a left left adjoint $L$, and we have 
\begin{displaymath}
\Ez{n}\cong\hocolim_{\pr/\zeta_n}(X\wedge^{\bbL} Y) \cong \hocolim_{J_n}\theta^*(X\wedge^{\bbL} Y),
\end{displaymath}
see (\ref{Ezetansecond}).
So, instead of the spectral sequence (\ref{sszetan}) we can compute the following spectral sequence 
\begin{displaymath}
H_p(J_n; F_q(\theta^*(X\wedge^{\bbL}Y))) \Rightarrow F_{p+q}(\hocolim_{J_n}\theta^*(X\wedge^{\bbL}Y)) 
\end{displaymath}
since both converge to the same target, \emph{i.e.}, the $F_*$-homology of $\Ez{n}$, 
\[F_*(\hocolim_{J_n}\theta^*(X\wedge^{\bbL}Y))\cong F_*(\hocolim_{\pr/\zeta_n} (X\wedge^{\bbL} Y))\cong F_*(E_{\zeta_n}).\]
In fact this, can be made stronger. The adjoint pair $L\colon \proz{n}\rightleftarrows J_n\colon \theta$ induces a natural isomorphism
\[H_*(\proz{n}; F_q(X\wedge^{\bbL}Y ))\cong H_*(J_n, \theta^* F_q(X\wedge^{\bbL}Y ) ) . 
\]

From the diagram $J_n$ we again only need to consider $F_{n}(-)$ and $F_{-n-1}(-)$. Firstly, we apply $F_{n}(-)$ to the diagram 
$\theta^*(X\wedge^{\bbL}Y)$ and we get $F_{n}(\theta^*(X\wedge^{\bbL}Y))\in \cA^{J_n}$ as below.
\begin{equation*}
\begin{tikzpicture}[scale=1.5]
\node (zero) at (0,1) {$Z^{(i)} \otimes\widetilde{Z}^{(j)}$};
\node (one) at (1,0) {$0$};
\node (two) at (2,1) {$Z^{(i+1)}\otimes\widetilde{Z}^{(i-1)}$};
\node (three)at ( 3,0) {$\ldots$};
\node (four) at (-1,0) {$0$};
\node (five) at (-2,1) {$Z^{(i-1)}\otimes\widetilde{Z}^{(i+1)}$};
\node (six) at (-3,0) {$\ldots$};
\draw [->] (one)-- (zero);
\draw [->] (one)-- (two);
\draw [->] (three)-- (two);
\draw [->] (four)-- (zero);
\draw [->] (four)-- (five);
\draw [->] (six)-- (five);
\end{tikzpicture}
\end{equation*}
From this we get that 
\begin{equation*}
H_0(J_n; F_{n}(\theta^*(X\wedge^{\bbL}Y)))= \bigoplus_{i+j=n}Z^{(i)}\otimes \widetilde{Z}^{(j)}
\end{equation*}
and 
\[
H_p(J_n; F_{n}(\theta^*(X\wedge^{\bbL}Y)))=0, \ \ p\geq 1.  
\]

Next, we will apply the functor $F_{n-1}(-)$ to obtain the diagram $F_{n-1}(\theta^*(X\wedge^{\bbL}Y))\in \cA^{J_n}$ depicted below.
\begin{equation*}
\begin{tikzpicture}[scale=1.5]
\node (zero) at (0,1) {$0$};
\node (one) at (1,0) {$B^{(i)}\otimes\widetilde{B}^{(j-1)}$};
\node (two) at (2,1) {$0$};
\node (three)at ( 3,0) {$\ldots$};
\node (four) at (-1,0) {$B^{(i-1)}\otimes \widetilde{B}^{(j)}$};
\node (five) at (-2,1) {$0$};
\node (six) at (-3,0) {$\ldots$};
\draw [->] (one)-- (zero);
\draw [->] (one)-- (two);
\draw [->] (three)-- (two);
\draw [->] (four)-- (zero);
\draw [->] (four)-- (five);
\draw [->] (six)-- (five);
\end{tikzpicture}
\end{equation*}
From the above we get that
\[
H_1(J_n; F_{n-1}(\theta^* (X\wedge^{\bbL}Y )))=\bigoplus_{i+j=n+1}B^{(i)}\otimes \widetilde{B}^{(j)}      \]
and 
\[
H_p(J_n; F_{n-1}(\theta^*(X\wedge^{\bbL}Y)))=0 \ \  \mbox{for $p=0$ and $p\geq 2$}.
\]
This completes the computation of the $E^2$-term of the final spectral sequence. It is concentrated in degrees $(0,m)$ and $(1,m-1)$ with $m \equiv\  n \operatorname{mod} N$. Therefore, the spectral 
sequence collapses and we 
have a short exact sequence
\begin{equation*}
0\lra \bigoplus_{i+j=n}Z^{(i)}\otimes \widetilde{Z}^{(j)}  \lra  F_{n}(E_{\zeta_n}) \lra \bigoplus_{i+j=n-1}B^{(i)}\otimes \widetilde{B}^{(j)}\lra 0. 
\end{equation*}

\noindent\fbox{%
    \parbox{24em}{%
       {\bf Step 3: the monomorphism $F_*(E_{\gamma_n}) \rightarrow F_*(E_{\zeta_n})$} 
    }%
}

\bigskip
Now that we have calculated both spectral sequences we can continue with the proof that $F_*(E_{\gamma_n}) \rightarrow F_*(E_{\zeta_n})$ is a monomorphism. The map of posets $\psi\colon \pr/\gamma_n\to \pr/\zeta_n$ induces morphisms on homologies of categories with coefficients $F_{n}(-)$ and $F_{n-1}(-)$ respectively, \ie,
\begin{align}
H_*(\pr/\gamma_n; F_{n}(X\wedge^{\bbL}Y))   &\lra H_*(\pr/\zeta_n; F_{n}(X\wedge^{\bbL}Y))\cong H_*(J_n; F_{n}(\theta^*(X\wedge^{\bbL}Y )))   \nonumber \\ 
H_*(\pr/\gamma_n; F_{n-1}(X\wedge^{\bbL}Y)) &\lra H_*(\pr/\zeta_n; F_{n-1}(X\wedge^{\bbL}Y))\cong H_*(J_n; F_{n-1}(\theta^*(X\wedge^{\bbL}Y ))).   \nonumber 
\end{align}
Hence we have a morphism of short exact sequences
\begin{equation*}
\xymatrix{
 0 \ar[r] &  \Oplus_{i+j=n}\left( Z^{(i)}\otimes \widetilde{B}^{(j)}\coprod_{B^{(i)}\otimes \widetilde{B}^{(j)}}B^{(i)}\otimes \widetilde{Z}^{(j)}\right) \ar[r] \ar[d] & F_{n}(E_{\gamma_n}) \ar[r] \ar[d] & \Oplus_{i+j=n+1}B^{(i)}\otimes 
\widetilde{B}^{(j)}  \ar[d]^{\cong} \ar[r] & 0 \\
0 \ar[r] &\Oplus_{i+j=n}Z^{(i)}\otimes \widetilde{Z}^{(j)}  \ar[r] & F_{n}(E_{\zeta_n}) \ar[r] &  \Oplus_{i+j=n+1}B^{(i)}\otimes \widetilde{B}^{(j)} \ar[r] & 0.}
\end{equation*}
By naturality, the left vertical map is the direct sum of the pushout-product maps 
\[ \lambda_i\boxprod \widetilde{\lambda}_j\colon  \left(Z^{(i)}\otimes \widetilde{B}^{(j)}\coprod_{B^{(i)}\otimes \widetilde{B}^{(j)}}B^{(i)}\otimes \widetilde{Z}^{(j)}\right)\lra Z^{(i)}\otimes 
\tZ^{(j)}.  \]
By Lemma \ref{ppinjective}, the map $\lambda_i\boxprod\widetilde{\lambda}_j$ is injective which means that so is the left vertical map. The five lemma now implies that the morphism 
\begin{displaymath}
 F_{n}(E_{\gamma_n})\lra F_{n}(E_{\zeta_n})
\end{displaymath}
is an injection. In particular, $F_*(E_{\gamma_n})$ and $F_*(\Ez{n})$ are concentrated in the correct degrees and the induced morphisms $F_*(E_{\gamma_n})\to F_*(E_{\zeta_n}) $ are injections. This 
concludes the proof of the proposition.
\end{proof}

Corollary \ref{pullbackviaiisine} now follows: the diagram $E$ is indeed in the subcategory $\cL \subseteq \Ho(\cM^{\cC_N})$ as the vertices are in the correct degree shifts of $\cB$, and $F$ applied to the edges $E_{\gamma_n} \to E_{\zeta_n}$ is a monomorphism, which is precisely how $\cL$ was defined. 

\subsection{Cones }\label{newcones}
In the previous section we proved that for any two crowned diagrams $X, Y\in \cL$ which are objectwise projective, $i^*E=i^*\pr_!(X\wedge^{\bbL} Y)\in \cL$. In this subsection we will prove that applying the functor $\cQ$ to the object $i^*E$ is a good model for the tensor product $\cQ(X)\otimes \cQ(Y)$. This will follow as a corollary from the following proposition.
\begin{prop}\label{newconesE}
Consider $E\in \bbL\pr_!(X\wedge^{\bbL} Y)\in \Ho(\cM^{\cD_N}) $ and let $i^*E$ be the pullback of $E$ along $i\colon \cC_N\lra \cD_N$. For every $n\in \bbZ/N\bbZ$ we have a canonical isomorphism
\begin{displaymath}
\operatorname{cone}(i^*E_{\beta_{n-1}}\lra i^*E_{\zeta_n})\cong \bigvee_{i+j=n}  \operatorname{cone}(k_i) \wedge^{\bbL} \operatorname{cone}(\tilde{k}_j),
\end{displaymath}
where $k_i: X_{\beta_{i-1}} \longrightarrow X_{\zeta_i}$ is a structure morphism of $X \in \cL$. 
\end{prop}

\begin{proof}

This proof has three main parts. Firstly, we will work out the morphism $i^*E_{\beta_{n-1}}\lra i^*E_{\zeta_n}$ by calculating the relevant values of $E_{\beta_{n-1}}$ and $E_{\zeta_n}$ using their description as homotopy colimits over slice categories, see Section \ref{sec:slicecats}. We will arrive at the conclusion that the LHS is actually $\hocolim_{\pro/\zeta_n}\left(\cone(\varepsilon_{X\wedge^{\bbL}Y})\right)$, where $\varepsilon_{X\wedge^{\bbL}Y}$ is the counit of a certain adjunction. We will then explicitly determine the map of diagrams $\varepsilon_{X\wedge^{\bbL}Y}$ in Step 2 and calculate its cone in Step 3. 

\bigskip
\noindent\fbox{%
    \parbox{18em}{%
       {\bf Step 1: Unravelling $i^*E_{\beta_{n-1}}\lra i^*E_{\zeta_n}$.}
    }%
}

\bigskip
Recall the slice categories of the map $\pr\colon \cC_N\times \cC_N\lra \cD_N$ over $\zeta_n$, $\pr/\zeta_n$ from Example (\ref{zetan}), and $\prog{n}$ from Example (\ref{gamman}). 
By definition of $i$, we have that 
\[
(i^*E_{\beta_{n-1}}\lra i^*E_{\zeta_n}) = E_{\gamma_{n-1}} \longrightarrow E_{\zeta_n}.
\]

Let us begin by recalling the diagram $X\wedge^{\bbL}Y \in \Ho(\cM^{\pr/\zeta_n}).$
The red color shows the image of the map of posets $\phi\colon \pr/\gamma_{n-1}\lra \pr/\zeta_n.$
 \begin{equation}\label{imageofgammanzetan}
\begin{tikzpicture}[scale=.7]
  \node (one) at (0,2) {$(\zeta_i ,\zeta_j)$};     
  \node (a) at (-4,-.5) {$(\beta_{i-1}, \zeta_j)$};
  \node (b) at (-1,0) {$(\zeta_i ,\beta_j)$};
  \node (c) at (1,0) {$(\beta_i, \zeta_j)$};
  \node (d) at (4,-0.5) {$(\zeta_i, \beta_{j-1})$};
  \node (zero) at (0,-2) {$(\beta_i, \beta_j)$}; 
	\node (minus) at (0,-4) {$(\beta_{i-1}, \beta_{j-1})$};
	\node (minusone) at  (-4.5,-2.5) {$(\beta_{i-1}, \beta_j)$};
	\node (minustwo) at  (4.5,-2.5) {$(\beta_i,  \beta_{j-1})$}; 
	\node (extraleftone) at (-7,-.5)  {$\ldots$};   
	\node (extralefttwo) at (-8,-2) {$\ldots$};
	\node (extrarightone) at (8,-2) {$\ldots$};
	\node (extrarighttwo) at (7,-.5) {$\ldots$};     
	\draw  [->]   (zero)--(b) ; 
	\draw [->]  (zero)--(c) ; 
\draw	 [->](b)-- (one) ;
\draw	 [->](c)--(one) ;
	\draw [red][->]  (minus)--(a); 
	\draw [red][->]   (minus)-- (d) ; 
\draw	 [red][->]  (minusone)--(a) ; 
	\draw [->]  (minusone) --(b) ; 
\draw	 [->]  (minustwo)--(c) ; 
	\draw [red][->] (minustwo)--(d); 
\draw  [->]	(a)--(one) node[midway,above left] {$$}   ;
\draw	 [->] (d)--(one) node[midway, above right] {$$} ;
\draw [red][->]  (minusone) --(extraleftone) ;
\draw [->] (minusone)--(extralefttwo) ;
\draw [->] (minustwo)-- (extrarightone);
\draw [red][->] (minustwo)-- (extrarighttwo);
\end{tikzpicture}
\end{equation}
Recall from \eqref{value2} that
\begin{displaymath}
E_{\zeta_n}= \hocolim( \pr/\zeta_n \stackrel{\pi}{\lra} \cC_N\times \cC_N \xrightarrow{X\wedge^{\bbL}Y} \cM),
\end{displaymath}
and we committed an abuse of notation by writing 
\begin{equation*}
E_{\zeta_n}= \hocolim_{\pr/\zeta_n}(X\wedge^{\bbL}Y)=\hocolim_{\pr/\zeta_n}\pi^*(X\wedge^{\bbL}Y).
\end{equation*} 
Also, recall from \eqref{gammanonezetan} that the morphism $E_{\gamma_{n-1}}\lra E_{\zeta_n}$ 
is 
the canonical morphism
\begin{equation*}
E_{\gamma_{n-1}} =\hocolim_{\pr/\gamma_{n-1}}\phi^*(X\wedge^{\bbL}Y)\lra \hocolim_{\pr/\zeta_n}(X\wedge^{\bbL}Y) = E_{\zeta_n}
\end{equation*}
induced by the map of posets $\phi\colon \pr/\gamma_{n-1}\lra \pr/\zeta_n$. 
The pullback functor 
\begin{equation*}
\phi^*\colon \Ho(\cM^{\pro/\zeta_n})\lra \Ho(\cM^{\pro/\gamma_{n-1}})
\end{equation*} has a 
left adjoint defined by the homotopy left Kan extension $\bbL\phi_!$, that is, 
\begin{equation*}\label{hoadjunction}
\bbL\phi_!\colon\Ho(\cM^{\pr/\gamma_{n-1}}) \rightleftarrows \Ho(\cM^{\pr/\zeta_n})\colon \phi^*.
\end{equation*}
The counit of the derived adjunction $\varepsilon\colon \bbL\phi_!\phi^*\lra \mathrm{Id}$ provides the canonical natural transformation
\begin{equation}\label{counitxby}
\varepsilon_{X\wedge^{\bbL}Y} \colon \bbL\phi_!\phi^*(X\wedge^{\bbL}Y )\lra X\wedge^{\bbL}Y.
\end{equation}
Lastly, since $\bbL\phi_!$ is a homotopy left Kan extension, there is a canonical isomorphism
\begin{equation*}
\hocolim_{\pr/\gamma_{n-1}} \phi^*(X\wedge^{\bbL}Y) \cong \hocolim_{\pr/\zeta_n} \bbL\phi_!\phi^*(X\wedge^{\bbL}Y).
\end{equation*}
Putting all this together means that the LHS of Proposition \ref{newconesE} is 
\[
\hocolim_{\pro/\zeta_n}\left(\cone(\varepsilon_{X\wedge^{\bbL}Y})\right).
\]

\bigskip
\noindent\fbox{%
    \parbox{15em}{%
       {\bf Step 2: Working out $\varepsilon_{X\wedge^{\bbL}Y}$}.
    }%
}

The underlying diagram $X\wedge^{\bbL}Y\in  \Ho(\cM^{\pr/\zeta_n})$ looks as follows.
\begin{equation}\label{functorprozetan}
\begin{tikzpicture}[baseline=(current  bounding  box.north)][scale=.7]
  \node (one) at (0,2) {$ X_{\zeta_i} \wedge Y_{\zeta_j}$};     
  \node (a) at (-4,-.5) {$X_{\beta_{i-1}}\wedge Y_{\zeta_j}$};
  \node (b) at (-1,0) {$X_{\zeta_i}\wedge Y_{\beta_j}$};
  \node (c) at (1,0) {$X_{\beta_i}\wedge Y_{\zeta_j}$};
  \node (d) at (4,-0.5) {$X_{\zeta_i}\wedge Y_{\beta_{j-1}}$};
  \node (zero) at (0,-2) {$X_{\beta_i}\wedge Y_{\beta_j}$}; 
	\node (minus) at (0,-4) {$X_{\beta_{i-1}}\wedge Y_{\beta_{j-1}}$};
	\node (minusone) at  (-4.5,-2.5) {$X_{\beta_{i-1}}\wedge Y_{\beta_j}$};
	\node (minustwo) at  (4.5,-2.5) {$X_{\beta_i}\wedge  Y_{\beta_{j-1}}$}; 
	\node (extraleftone) at (-7,-.5)  {$\ldots$};   
	\node (extralefttwo) at (-8,-2) {$\ldots$};
	\node (extrarightone) at (8,-2) {$\ldots$};
	\node (extrarighttwo) at (7,-.5) {$\ldots$};     
	\draw  [->]   (zero)--(b) ; 
	\draw [->]  (zero)--(c) ; 
\draw	 [->](b)-- (one) ;
\draw	 [->](c)--(one) ;
	\draw [->]  (minus)--(a); 
	\draw [->]   (minus)-- (d) ; 
\draw	 [->]  (minusone)--(a) ; 
	\draw [->]  (minusone) --(b) ; 
\draw	 [->]  (minustwo)--(c) ; 
	\draw [->] (minustwo)--(d); 
\draw  [->]	(a)--(one) node[midway,above left] {$$}   ;
\draw	 [->] (d)--(one) node[midway, above right] {$$} ;
\draw [->]  (minusone) --(extraleftone) ;
\draw [->] (minusone)--(extralefttwo) ;
\draw [->] (minustwo)-- (extrarightone);
\draw [->] (minustwo)-- (extrarighttwo);
\end{tikzpicture}
\end{equation}
Furthermore, the homotopy left Kan extension $\bbL\phi_!\phi^*(X\wedge^{\bbL}Y)\in \Ho(\cM^{\pro/\zeta_n})$ is the following.
\begin{equation}\label{hokanextendofabove}
\begin{tikzpicture}[baseline=(current  bounding  box.north)][scale=.9]
  \node (one) at (0,2) {$ X_{\zeta_{i-1}}\wedge Y_{\zeta_j} \coprod^h_{X_{\beta_{i-1}}\wedge Y_{\beta_{j-1}}}X_{\zeta_i}\wedge Y_{\beta_{j-1}}  $};     
  \node (a) at (-4,-.5) {$X_{\beta_{i-1}}\wedge Y_{\zeta_j}$};
  \node (b) at (-1,0) {$X_{\beta_{i-1}}\wedge Y_{\beta_j}$};
  \node (c) at (1,0) {$X_{\beta_i}\wedge Y_{\beta_{j-1}}$};
  \node (d) at (4,-0.5) {$X_{\zeta_i}\wedge Y_{\beta_{j-1}}$};
  \node (zero) at (0,-2) {$\ast$}; 
	\node (minus) at (0,-4) {$X_{\beta_{i-1}}\wedge Y_{\beta_{j-1}}$};
	\node (minusone) at  (-4.5,-2.5) {$X_{\beta_{i-1}}\wedge Y_{\beta_j}$};
	\node (minustwo) at  (4.5,-2.5) {$X_{\beta_i}\wedge  Y_{\beta_{j-1}}$}; 
	\node (extraleftone) at (-7,-.5)  {$\ldots$};   
	\node (extralefttwo) at (-8,-2) {$\ldots$};
	\node (extrarightone) at (8,-2) {$\ldots$};
	\node (extrarighttwo) at (7,-.5) {$\ldots$};     
	\draw  [->]   (zero)--(b) ; 
	\draw [->]  (zero)--(c) ; 
\draw	 [->](b)-- (one) ;
\draw	 [->](c)--(one) ;
	\draw [->]  (minus)--(a); 
	\draw [->]   (minus)-- (d) ; 
\draw	 [->]  (minusone)--(a) ; 
	\draw [double equal sign distance]  (minusone) --(b) ; 
\draw	 [double equal sign distance]  (minustwo)--(c) ; 
	\draw [->] (minustwo)--(d); 
\draw  [->]	(a)--(one) node[midway,above left] {$$}   ;
\draw	 [->] (d)--(one) node[midway, above right] {$$} ;
\draw [->]  (minusone) --(extraleftone) ;
\draw [->] (minusone)--(extralefttwo) ;
\draw [->] (minustwo)-- (extrarightone);
\draw [->] (minustwo)-- (extrarighttwo);
\end{tikzpicture}
\end{equation}
We briefly explain how we calculated the left homotopy Kan extension $\bbL\phi_!(X\wedge^{\bbL}Y)$. From the formula \eqref{holanformula} of calculating homotopy Kan extensions, we can calculate
the homotopy left Kan extension $\bbL\phi_!\phi^*$ at an object $(\alpha_s,\alpha_t)\in \proz{n}$ as 
\[ \bbL\phi_!(X\wedge^{\bbL}Y)_{(\alpha_s,\alpha_t)}\cong \hocolim( \phi/(\alpha_s,\alpha_t) \stackrel{\pi}{\rightarrow} \prog{n-1}\xrightarrow{\phi^*(X\wedge^{\bbL} Y) } \cM).  \]
For the object $(\zeta_i,\beta_j)$, the slice $\phi/(\zeta_i,\beta_j)$ consists only of the object the object $(\beta_{j-1},\beta_j)$, which implies that
\[ (\bbL\phi_!)_{(\zeta_i,\beta_j)}=\Xb{i-1}\wedge \Yb{j}.  \]
For the object $(\beta_i,\zeta_j),$ the argument is the same as above. For $(\beta_i,\beta_j),$ the slice category $\phi/(\beta_i,\beta_j)$ is empty, which means that
\[(\bbL\phi_!)_{(\beta_i,\beta_j)}\cong \ast.   \]
For the object $(\zeta_i,\zeta_j),$ the slice category $ \phi/(\zeta_i,\zeta_j)$
is the poset 
\[
\xymatrix{     & (\beta_{i-1},\zeta_j) &   & (\zeta_i,\beta_{j-1}) &    \\
              (\beta_{i-1},\beta_j)\ar[ur]&   & (\beta_{i-1},\beta_{j-1}) \ar[ur] \ar[ul] &   & (\beta_i,\beta_{j-1}). \ar[ul] }
\]
But the subposet
\[\xymatrix{  (\beta_{i-1},\zeta_j)     & (\beta_{i-1},\beta_{j-1}) \ar[r] \ar[l] &    (\zeta_i,\beta_{j-1}) }
\]
is homotopy final, which means that

\[ (\bbL\phi_!)_{(\zeta_i,\zeta_j)} \cong
(\bbL\phi_!)_{(\zeta_i,\zeta_j)}\cong \hocolim \left(  \begin{tikzcd}
  \Xb{i-1}\wedge \Yb{j-1} \arrow[r,"k_i\wedge 1"]   \arrow[d,"1\wedge \tk_j",swap]    &  \Xz{i}\wedge \Yb{j-1}\\ 
 \Xb{i-1}\wedge \Yz{j}                            &
\end{tikzcd}\right).
\]

\bigskip
\noindent\fbox{%
    \parbox{20em}{%
       {\bf Step 3: Calculating the cone in the LHS.}
    }%
}

\bigskip
Next, we calculate the cone of the natural transformation $\varepsilon_{X\wedge^{\bbL}Y}$ \eqref{counitxby} of diagrams in $\Ho(\cM^{\proz{n}})$. We have the diagram
$\cone(\varepsilon_{X\wedge^{\bbL}Y})\in \Ho(\cM^{\proz{n}}),$ which is 
\begin{align*}
\cone\left(\varepsilon_{X\wedge^{\bbL}Y}\right) \colon \pro/\zeta_n &\lra \cM\\
                                                (\alpha_s,\alpha_t) &\mapsto \cone\left(\phi_!(X\wedge^{\bbL}Y)_{(\alpha_s,\alpha_t)} \lra (X\wedge^{\bbL}Y)_{(\alpha_s,\alpha_t)} \right).
\end{align*}
In other words, we are taking objectwise cones of the canonical map from the diagram \eqref{hokanextendofabove} to the diagram \eqref{functorprozetan}. 
This means that $\cone\left(\varepsilon_{X\wedge^{\bbL}Y}\right)$ is as below.
\begin{equation}\label{conezndia}
\begin{tikzpicture}[scale=.8]
  \node (one) at (0,2) {$ \mathrm{cone}(k_i\boxprod^{\bbL} \widetilde{k}_j) $};     
  \node (a) at (-4,-.5) {$\ast$};
  \node (b) at (-1,0) {$C^i\wedge Y_{\beta_j}$};
  \node (c) at (1,0) {$X_{\beta_i}\wedge \widetilde{C}^j $};
  \node (d) at (4,-0.5) {$\ast$};
  \node (zero) at (0,-2) {$\Sigma X_{\beta_i}\wedge Y_{\beta_j}$}; 
	\node (minus) at (0,-4) {$\ast$};
	\node (minusone) at  (-4.5,-2.5) {$\ast$};
	\node (minustwo) at  (4.5,-2.5) {$\ast$}; 
	\node (extraleftone) at (-7,-.5)  {$\ldots$};   
	\node (extralefttwo) at (-8,-2) {$\ldots$};
	\node (extrarightone) at (8,-2) {$\ldots$};
	\node (extrarighttwo) at (7,-.5) {$\ldots$};     
	\draw  [->]   (zero)--(b) ; 
	\draw [->]  (zero)--(c) ; 
\draw	 [->](b)-- (one) ;
\draw	 [->](c)--(one) ;
	\draw [->]  (minus)--(a); 
	\draw [->]   (minus)-- (d) ; 
\draw	 [->]  (minusone)--(a) ; 
	\draw [->]  (minusone) --(b) ; 
\draw	 [->]  (minustwo)--(c) ; 
	\draw [->] (minustwo)--(d); 
\draw  [->]	(a)--(one) node[midway,above left] {$$}   ;
\draw	 [->] (d)--(one) node[midway, above right] {$$} ;
\draw [->]  (minusone) --(extraleftone) ;
\draw [->] (minusone)--(extralefttwo) ;
\draw [->] (minustwo)-- (extrarightone);
\draw [->] (minustwo)-- (extrarighttwo);
\end{tikzpicture}
\end{equation}
Here, we have denoted 
\begin{align*} C^i:=\cone(k_i)=\mathrm{cone}(X_{i-1}\lra X_{\zeta_i})  \\
 \widetilde{C}^j:=\cone(\widetilde{k}_j)=\mathrm{cone}(Y_{j-1}\lra Y_{\zeta_i}).
\end{align*}
Next, we determine the homotopy colimit of the diagram $\cone\left(\varepsilon_{X\wedge^{\bbL}Y}\right)$. One way is to observe that the homotopy colimit of the above diagram is isomorphic in $\Ho(\cM)$ to the homotopy colimit of (finite) coproduct of squares
\begin{equation}\label{broken}
\xymatrix{
\Sigma \Xb{i}\wedge \Yb{j} \ar[d] \ar[r] & \cone(k_i)\wedge \Yb{j} \ar[d] \\
  \Xb{i}\wedge \cone(\tk_j) \ar[r] & \cone(k_i\boxprod^{\bbL} \tk_j),   
}
\end{equation}
where we can consider the above as an object in $\Ho(\cM^{[1] \times [1]})$. Formally this is obtained by taking the visually obvious map of posets $f\colon [1] \times [1]\lra \proz{n}$ (i.e. $(0,0) \mapsto (\beta_i, \beta_j), (0,1) \mapsto (\zeta_i, \beta_j), (1,0) \mapsto (\beta_i, \zeta_j), (1,1) \mapsto(\zeta_i, \zeta_j)$) and considering the pullback
\[f^*\colon \Ho(\cM^{\proz{n}})\lra \Ho(\cM^{[1] \times [1]})  .\]
The bottom right corner of the poset $[1] \times [1]$ is its final object, which implies that the homotopy
colimit of the diagram \eqref{broken} is naturally isomorphic to $\cone(k_i\boxprod^{\bbL} \tk_j)$. Hence the homotopy colimit over $\pr/\zeta_n$ is, up to natural isomorphism, the 
coproduct $\bigvee_{i+j=n}\cone(k_i\boxprod^{\bbL} \cone(\tk_j))$. Another way of seeing this is by pulling back the above diagram to $\theta_n\colon J_n\lra \pro/\zeta_n$, and we get the diagram 
\[
\begin{tikzcd}[column sep=tiny]
& \cone(k_{i-1}\boxprod^{\bbL} k_{j+1}) & & \cone(k_i\boxprod^{\bbL} k_j) & & \cone(k_{i+1}\boxprod^{\bbL} k_{j-1}) & \\
\ldots \arrow[ur]&             &\arrow[ul] \ast \ar[ur]  &         &\arrow[ul] \ast \arrow[ur]  &             &  \ldots \ar[ul].
\end{tikzcd}
\]
All in all, we have that the homotopy colimit of the diagram \eqref{conezndia} is 
\begin{equation}\label{hocolimofcone}
\hocolim_{\pro/\zeta_n}\left(\cone( \varepsilon_{X\wedge^{\bbL}Y}) \right)\cong \bigvee_{i+j=n}\cone(k_i\boxprod^{\bbL}\widetilde{k}_j).
\end{equation}
Finally, by Corollary \ref{conemonoidal}, we have the canonical isomorphism 
\[\cone(k_i\boxprod^{\bbL} \widetilde{k}_j)\cong \cone(k_i)\wedge^{\bbL} \cone(\tk_j)\]
for each pair $i,j\in \bbZ/N\bbZ$.
 The coproduct of these isomorphisms together with \eqref{hocolimofcone} gives us that 
\[\hocolim_{\pro/\zeta_n}\left(\cone(\varepsilon_{X\wedge^{\bbL}Y})\right)\cong \bigvee_{i+j=n} \cone(k_i)\wedge^{\bbL} \cone(\tk_j).   \]
Let us now gather all this information to prove Proposition \ref{newconesE}. Calculating the homotopy cofiber (cone) of the morphisms $i^*E_{\beta_{n-1}}\lra i^*E_{\zeta_n}$ is the same thing as calculating the homotopy cofiber
$E_{\gamma_{n-1}}\lra E_{\zeta_n}$. We have the following natural isomorphisms.
\begin{align}
\cone\left( i^*E_{\beta_{n-1}}\lra i^*E_{\zeta_n}\right) &= \cone\left(E_{\gamma_{n-1}}\lra E_{\zeta_n}  \right)  \nonumber  \\
                                                      &= \cone\left( \hocolim_{\pr/\gamma_{n-1}}\phi^*(X\wedge^{\bbL}Y) \lra \hocolim_{\pro/\zeta_n}(X\wedge^{\bbL}Y)  \right) \nonumber \\
																											&\cong \cone\left(\hocolim_{\pr/\zeta_n}\phi_!\phi^*(X\wedge^{\bbL}Y)  \lra \hocolim_{\pr/\zeta_n}(X\wedge^{\bbL}Y)   \right) \nonumber \\
																											&\cong \hocolim_{\pro/\zeta_n}\left(\cone\left( \phi_!\phi^*(X\wedge^{\bbL}Y)   \lra (X\wedge^{\bbL}Y) \right)  \right) \nonumber \\
																											&\cong \bigvee_{i+j=n} \cone(k_i\boxprod^{\bbL} \widetilde{k}_j) \nonumber \\
																											&\cong  \bigvee_{i+j=n} \cone(k_i)\wedge^{\bbL} \cone(\widetilde{k}_j). \nonumber 
\end{align}
\end{proof}

\begin{cor}\label{homoofnewcones}
Let $X,Y,E$ as before and assume furthermore that $F_*(\Xa{n}), F_*(\Ya{n})\in \cAp$ for any $n\in \bbZ/N\bbZ$ and any $\alpha\in \left\{\zeta,\beta\right\}$. Then there is a canonical isomorphism 
\begin{displaymath}
C^{(n)}(i^*E)= F_*(\operatorname{cone}(i^*E_{\beta_{n-1}}\lra i^*E_{\zeta_n}))\cong \bigoplus_{i+j=n}C^{(i)}(X)\otimes C^{(j)}(Y).
\end{displaymath}
\end{cor}
\begin{proof}
By our assumption, for any $\alpha\in \left\{\zeta,\beta\right\}$ and any $n\in \bbZ/N\bbZ$, the object $F_*\Xa{n}$ is projective. Therefore, by definition, $Z^{(n)}(X)$ and $B^{(n-1)}(X)$ 
are projective. The short exact sequence \eqref{sexci} now implies that for any $i\in\bbZ/N\bbZ$ the graded object $C^{(i)}(X)$ is projective. It follows by our assumptions that
\[F_*(\cone k_s \wedge^{\bbL} \cone \tk_t) \cong F_*(\cone k_s)\otimes F_*(\cone k_t)   .\]
By Proposition \ref{newconesE} we have 
\[\cone(i^*E_{\beta_{n-1}}\lra i^*E_{\zeta_n})\cong \bigvee_{i+j=n}  \cone(k_i) \wedge^{\bbL} \cone(\tilde{k}_j),\]
and applying the functor $F_*(-)$ we have
\begin{align*}
F_*( \operatorname{cone}(i^*E_{\beta_{n-1}}\lra i^*E_{\zeta_n}))&\cong F_*\left(\bigvee_{i+j=n}  \operatorname{cone}(k_i) \wedge^{\bbL} \operatorname{cone}(\tilde{k}_j)\right) \\
                                                                                &\cong\bigoplus_{i+j=n}F_*(\cone(k_i)\wedge^{\bbL} \cone \tk_j)  \\
																																								&\cong\bigoplus_{i+j=n} F_*(\cone k_i) \otimes F_*(\cone \tk_j ).   \\
																																								\end{align*}
Shifting the above by $[n]=[i+j]$ we have  
\begin{equation*}
C^{(n)}(i^*E)\cong \bigoplus_{i+j=n} C^{(i)}(X)\otimes C^{(j)}(Y).
\end{equation*}
\end{proof}

\subsection{Differentials}\label{differentials}
In the previous subsection we proved that $C_*(i^*E)\cong C_*(X)\otimes C_*(Y)$ as objects in $\cA$, so the diagram $i^*E$ is a good candidate for the tensor product $$C_*(X)\otimes C_*(Y).$$
The final step in order to show that indeed
\[\cQ(i^*E) \cong \cQ(X) \otimes  \cQ(Y )\]
as objects in $\twc$ is to prove that that the differential $d\colon C_*(i^*E) \to C_*(i^*E)[1]$ coincides with the differential of
the tensor product $C_*(X) \otimes C_*(Y)$. That is, we have to show that
\[(C_*(i^*E),d) \cong ((C_*(X)\otimes C_*(Y), d_{\otimes})   \]
where $d_{\otimes}$ is the differential of the tensor product of the dg-objects $(C_*(X),d)$ and $(C_*(Y),d)$.

\subsubsection{Reduction to the Case of Disks}
We will reduce the proof to a much simpler case. Let $L_*\in \twc$ and choose $s\in \bbZ/N\bbZ$. Without loss of generality we will assume that $L_*$ is degreewise projective. Consider the following map of differential graded objects
\begin{equation}\label{diskchain}
\xymatrix{ \ldots \ar[r] & 0 \ar[r] \ar[d] & L_s \ar@{=}[r] \ar@{=}[d] & L_s \ar[r] \ar[d]^{d^s} & 0 \ar[r] \ar[d] & \ldots \\
          \ldots  \ar[r]_{d^{s+2}} &  L_{s+1}   \ar[r]_{d^{s+1}}    &  L_s  \ar[r]_{d^s}        & L_{s-1}  \ar[r]_{d^{s-1}} & L_{s-2} \ar[r]_{d^{s-2}} & \ldots }
\end{equation}
where we view the top differential graded object as an object in $\cB[s-1]\oplus \cB[s]$, meaning that it is concentrated in degrees $s-1$ and $s$ modulo $N$. We denote this by $D^s(L_s)$, and we denote the above map of differential graded objects by
$f_{L,s}\colon D^s(L_s)\lra L_*$. Under the equivalence of categories $\cQ\colon \cL \lra \twc$ there 
are crowned diagrams $X$ and $X^{\prime}$ and a morphism $F\colon X\lra X^{\prime}$ such that the morphism $f_{L,s}$ is realized as $\cQ(F)$. This means that there are isomorphisms 
\[ \cQ(X)\cong D^s(L_s),\  \ \ \ \ \cQ(X^{\prime})\cong L_*   \]
and the following  diagram commutes.
\[
\begin{tikzcd}
 \cQ(X) \arrow{r}{\cQ(F)} \arrow[swap]{d}{\cong} & \cQ(X^{\prime}) \arrow{d}{\cong} \\
               D^s(L_s)     \arrow[swap]{r}{f_{L,s}}     &  L_*     
\end{tikzcd}
\]
Now let $M_*$ be another differential graded object, which we also assume to be degreewise projective, and let $t\in \bbZ/N\bbZ$. Similarly to \eqref{diskchain} we have the morphism
\[g_{M,t}\colon D^t(M_t)\lra M_*.\] Again, under the equivalence $\cQ$ there are crowned diagrams $Y$ and $Y^{\prime}$ and a morphism $G\colon Y\lra Y^{\prime}$ such 
that
 \[ \cQ(Y)\cong D^t(M_t), \ \  \cQ(Y^{\prime})\cong M_*   \]
and the following diagram commutes.
\[\begin{tikzcd}
\cQ(Y) \arrow{r}{\cQ(G)} \arrow[swap]{d}{\cong} & \cQ(Y^{\prime}) \arrow{d}{\cong} \\
               D^t(M_t)     \arrow[swap]{r}{g_{M,t}}     &  M_*       
\end{tikzcd}
\]    
We have the morphism of dg objects
$f_{L,s}\otimes g_{M,t}\colon D^s(L_s)\otimes D^t(M_t)\lra L_*\otimes M_* $
which is
\begin{equation}\label{mapoftensors}
\xymatrix{ \ldots \ar[r] & L_s\otimes M_t \ar[r] \ar[d]  &  (L_s\otimes M_t)\oplus (L_s\otimes M_t) \ar[r] \ar[d]^{(d^s\otimes\mathrm{id},\mathrm{id}\otimes \widetilde{d}^t)} & \ldots \\
              \ldots   \ar[r]        &   \Oplus_{i+j=n}L_i\otimes M_j  \ar[r]  & \Oplus_{i+j=n+1} L_i\otimes M_j   \ar[r]    &\ldots }
\end{equation}
where the left vertical morphism is the inclusion of the $(s,t)$th summand and the right vertical map is the universal map out of the coproduct. 

Now assume that
\[\cQ(i^*\pr_!(X\wedge^{\bbL} Y))\cong \cQ(X)\otimes \cQ(Y),    \]
that is, we prove our claim for the case of $X\cong \cQ^{-1}(D^s(L_s))$ and  $Y \cong \cQ^{-1}(D^t(M_t))$. The commutativity of the square  \eqref{mapoftensors} implies that the bottom vertical maps must also coincide with the tensor product $L_*\otimes M_*$, \ie,
\[ Q(i^*\pr_!(X^{\prime}\wedge^{\bbL} Y^{\prime}))\cong L_*\otimes M_* \]
and the following diagram commutes degreewise.
\[\begin{tikzcd}
Q(i^*\pr_!(X\wedge^{\bbL} Y) ) \arrow{r} \arrow{d} &  \cQ(i^*\pr_!(X^{\prime}\wedge^{\bbL} Y^{\prime})) \arrow{d} \\
D^s(L_s)\otimes D^t(M_t)   \arrow{r}                &   L_*\otimes M_*  
\end{tikzcd}
\]
The horizontal maps are indeed maps of dg-objects, so if we can show that the left hand vertical map is too, then the claim follows for the general $L_*$ and $M_*$. The proof of the former will occupy the next subsection. 

\subsection{Differentials for Disks}

To prove the claim for disks, we discuss a 
crowned diagram that corresponds to the disks.
By the \cite[Proposition 3.2.1]{PA12}, 
there is an object $A\in \Ho(\cM)$, such that $F_*A\in \cB[s-1]=L_s,$ which is due to the fact that our assumptions force the corresponding Adams spectral sequence to collapse.  
Consider the following crowned diagram 
\[
X= \begin{gathered}\xymatrix{&     &  \ast                &   A                    &  \ast   & \ldots   \\
&\ldots \ar[ur] & \ast \ar[u]  \ar[ur] &  A \ar@{=}[u] \ar[ur]  &  \ast \ar[u]  \ar[ur]    &              
}\end{gathered}
\]
where the non-trivial entries are at the $(s-1)$-spot, \ie, 
\[
X_{\beta_{s-1}}=X_{\zeta_{s-1}}=A.
\]
The diagram $X$ is in $\cL$ since
\[B_*(X)=B^{(s-1)}(X)=F_*\Xb{s-1}=F_*A \,\,\,\,\mbox{and}\,\,\,\,Z_*(X)=Z^{(s-1)}(X)= F_*\Xz{s-1}=F_*A  .\]
Next, we calculate $(C_*(X),d)\in \twc$. The only nontrivial cones are $\cone(k_{s-1})$ and $\cone(k_s)$. This means that
\begin{align*}
C^{(s)}(X)&= F_*\cone(k_s)= F_* \cone( A\lra \ast )= F_*(\Sigma A)\cong (F_*A)[1],   \\
C^{(s-1)}(X)&= F_* \cone(k_{s-1})= F_*\cone(*\lra A)=F_*A, \\
C_*(X)&=C^{(s-1)}(X)\oplus C^{(s)}(X).
\end{align*}
We obtain that
$\lambda\colon B_*(X)\to Z_*(X)$ is the identity map, $\iota\colon Z_*(X)\to C_*(X) $ is inclusion to the first factor and $\rho\colon C_*(X)\to B_*(X)$ is the projection to the second factor.
It follows that $d\colon C_*(X)\to C_*(X)[1]$ is the identity. Similarly, $D^t(L_t)$ is mapped to a crowned diagram $Y$ in which 
\[\Yb{t-1}= \Yz{t-1}=\tA\] where 
the only non-trivial morphism is the identity. 

We now have the ingredients to deal with the following proposition. 
\begin{prop}\label{differentialsdisks}
Let $X$ and $Y$ be the objects of $\cL$ of the form $\mathcal{Q}^{-1}(D^s(L_s))$ and $\mathcal{Q}^{-1}(D^t(M_t))$. Then
\[ \cQ(i^*\bbL\pr_!(X\wedge^{\bbL}Y)) \cong \left(C_*(X)\otimes C_*(Y),d_{\otimes} \right)  \]
where $(C_*(X)\otimes C_*(Y), d_{\otimes})$ is the tensor product of $C_*(X)$ and $C_*(Y)$ in $\twc$.
\end{prop}

\begin{proof}

We note that the tensor product $D^s(L_s)\otimes D^t(M_t)$ is concentrated in degrees $s+t, s+t-1,$ and  $s+t-2$ modulo $N$. As we already know that our chain complexes agree degreewise, these are the only degrees where we have to calculate our differential. As usual, we write $E=\bbL\pr_!(X\wedge^{\bbL}Y).$

We will work out the differential in the chain complex $\cQ(i^*\bbL\pr_!(X\wedge^{\bbL}Y))$, beginning with
\[
\cQ(i^*E)^{s+t} = C^{(s+t)}(i^*E)\to C^{(s+t-1)}(i^*E)=\cQ(i^*E)^{s+t-1},
\]
and we will discuss the other degree $$C^{(s+t-1)}(i^*E)\to C^{(s+t-2)}(i^*E) $$ afterwards. Our proof is divided into the following steps.

We start by going through the definition $\cQ$ applied to $i^*E$ using the descriptions given in Section \ref{Frankefunctor}, where we will arrive at the exact triangle
\[
\cone(\widehat{k}_{s+t}) \lra \Sigma E_{\gamma_{s+t-1}}\lra \Sigma E_{\zeta_{s+t-1}}\lra \Sigma\mathrm{cone}(\widehat{k}_{s+t-1}).
\]
The next steps separately determine $E_{\zeta_{s+t-1}}$ followed by the maps $\cone(\widehat{k}_{s+t})\lra \Sigma \Eg{s+t-1}$, $\Eg{s+t-1}\lra \Ez{s+t-1}$ and $\Ez{s+t-1}\lra \cone(\widehat{k}_{s+t-1})$. Putting those together, we obtain the desired differential.

\bigskip
\noindent\fbox{%
    \parbox{35em}{%
       {\bf Step 1: Recalling the construction of} $\cQ(i^*E)^{s+t} \longrightarrow \cQ(i^*E)^{s+t-1}$.
    }%
}

\bigskip
 By Proposition \ref{propA} and Proposition \ref{newconesE}
we can construct a diagram $E=\bbL\pr_!(X\wedge^{\bbL}Y)\in \Ho(\cM^{\cD_N})$ such that 
\begin{equation*}
i^*E\in \cL \ \   \mbox{and} \ \    \cone(i^*E_{\beta_{n-1}}\lra i^*E_{\zeta_n})\cong \bigvee_{i+j=n} \cone(k_i)\wedge \cone(\widetilde{k}_j).
\end{equation*} 
For notational convenience we will write 
\begin{equation*}
\widehat{k}_n\colon i^*\Eb{n-1}\lra i^*\Ez{n} \ \ \text{and} \ \ \widehat{l}_n\colon i^*\Eb{n}\lra \Ez{n}
\end{equation*}
for the structure maps of the crowned diagram $i^*E$. We briefly recall the construction of the differential 
\[ d\colon C_*(i^*E)\lra  C_*(i^*E)[1] \ \ d=\iota[1]\lambda[1]\rho C_*(i^*E)\lra C_*(i^*E)[1]. \]
Degreewise, the differential on $C^{(n)}(i^*E)\to C^{(n-1)}(i^*E)$ is given by applying $F_*(-)$ to the sequence of maps
\begin{equation}\label{compositionforiE}
\cone(\widehat{k}_n) \lra \Sigma E_{\gamma_{n-1}}\lra \Sigma E_{\zeta_{n-1}}\lra \Sigma\mathrm{cone}(\widehat{k}_{n-1}).
\end{equation}
Therefore, we have to show that for $n=s+t$ the sequence of maps \eqref{compositionforiE} after applying $F_*(-)$ gives us the differential of the tensor product of disks. By Proposition 
\ref{newconesE}, 
we have
\begin{align*}
\cone(\widehat{k}_{s+t})&\cong \cone(k_s)\wedge^{\bbL} \cone(\tk_t)    \\
\cone(\widehat{k}_{s+t-1})&\cong \left(\cone(k_{s-1})\wedge^{\bbL} \cone(\tk_{t})\right) \vee \left(\cone(k_{s})\wedge^L \cone(\tk_{t-1})\right).
\end{align*}
Recall that $A=X_{\beta_{s-1}}=X_{\zeta_{s-1}}$ and $\tA=Y_{\beta_{s-1}}=Y_{\zeta_{s-1}}$ as before.
Directly from the structure morphisms of the crowned diagrams $X$ and $Y$ we have
 \begin{align*}
\cone(\widehat{k}_{s+t}) &\cong (\Sigma A) \wedge (\Sigma\tA)  \\
\cone(\widehat{k}_{s+t-1})&\cong  \left(A\wedge \Sigma\tA\right) \vee \left(\Sigma A\wedge \tA \right).
\end{align*}
To analyse the sequence of maps \eqref{compositionforiE} it remains to calculate $\Eg{s+t-1}, \Ez{s+t-1}, \Ez{s+t}$ and the maps $E_{\gamma_{s+t-1}}\to E_{\zeta_{s+t-1}}$.
The maps
\begin{align}
\cone(\widehat{k}_{s+t})&\lra \Sigma \Eg{s+t-1} \label{prwtomap}   \\
\Sigma \Ez{s+t-1}         &\lra \Sigma\mathrm{cone}(\widehat{k}_{s+t-1}) \label{deuteromap}
\end{align}
are the canonical maps that are given by construction of distinguished triangles in a simplicial stable model category. The map \eqref{prwtomap} is the canonical map 
\[\cone(\widehat{k}_{s+t})\lra S^1\wedge E_{\gamma_{s+t-1}}\] 
see Definition \ref{elementary}. Similarly, the map 
\eqref{deuteromap} is the suspension of the canonical map
\[\Ez{s+t-1}\lra \cone(\widehat{k}_{s+t-1},)\]
see Definition \ref{elementary}.

\bigskip
\noindent\fbox{%
    \parbox{15em}{%
       {\bf Step 2: Calculating $E_{\zeta_n}$.}
    }%
}

\bigskip
To compute the above, let us recall from (\ref{subposetzetan}) the poset $J_n$ with inclusion $\theta\colon J_n\hookrightarrow \pr/\zeta_n$ and left adjoint $L\colon \pr/\zeta_n\to J_n$. 
For $i+j\equiv n$ modulo $N$ the poset $J_n$ looks as below.
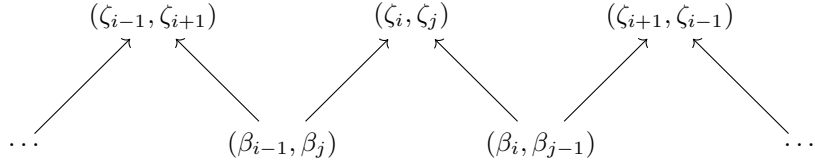
\begin{equation*}
\begin{tikzpicture}[scale=1.7]
\node (zero) at (0,1) {$(\zeta_i,\zeta_j)$};
\node (one) at (1,0) {$(\beta_i,\beta_{j-1})$};
\node (two) at (2,1) {$(\zeta_{i+1},\zeta_{i-1})$};
\node (three)at ( 3,0) {$\ldots$};
\node (four) at (-1,0) {$(\beta_{i-1},\beta_j)$};
\node (five) at (-2,1) {$(\zeta_{i-1},\zeta_{i+1})$};
\node (six) at (-3,0) {$\ldots$};
\draw [->] (one)-- (zero);
\draw [->] (one)-- (two);
\draw [->] (three)-- (two);
\draw [->] (four)-- (zero);
\draw [->] (four)-- (five);
\draw [->] (six)-- (five);
\end{tikzpicture}
\end{equation*}
Also, recall from  (\ref{gamman}) the slice category $\pr/\gamma_n$, which for $i+j\equiv n$ looks as follows.
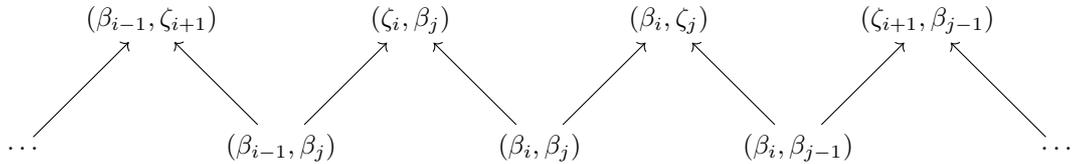
\begin{equation*}
\begin{tikzpicture}[scale=1.7]
\node (zero)at (0,0) {$(\beta_i,\beta_j)$};
\node (one) at (1,1) {$(\beta_i,\zeta_j)$};
\node (two) at (-1,1){$(\zeta_i,\beta_j)$};
\node (three) at (2,0){$(\beta_i,\beta_{j-1})$};
\node (four)  at (-2,0){$(\beta_{i-1},\beta_j)$};
\node (five) at  (3,1){$(\zeta_{i+1},\beta_{j-1})$} ;
\node (six) at   (-3,1) {$(\beta_{i-1},\zeta_{i+1})$};
\node (seven) at  (4,0) {$\ldots$};
\node (eight) at (-4,0) {$\ldots$};
\draw [->] (zero)-- (one) ;
\draw [->] (zero)-- (two);
\draw [->](three)-- (one);
\draw[->] (four)-- (two);
\draw [->](three)-- (five);
\draw [->](four)--(six);
\draw [->](seven)-- (five);
\draw [->](eight)--(six);
\end{tikzpicture}
\end{equation*}
By definition of homotopy left Kan extensions, we have
\begin{align*}
E_{\zeta_{n}}&=\hocolim_{\pr/\zeta_{n}} X\wedge^{\bbL} Y\cong \hocolim_{J_{n}} \theta^*(X\wedge^{\bbL}Y), \\
E_{\gamma_{n}}&=\hocolim_{\pr/\gamma_{n}}(X\wedge^{\bbL}Y).
\end{align*}
The maps 
\[\Eg{n-1}\lra \Ez{n} \ \ \text{and} \ \ \Eg{n-1}\lra \Ez{n-1} \]
are the maps of homotopy colimits induced by the respective map of posets 
\[\psi\colon \pro/\gamma_{n-1}\lra \pro/\zeta_{n} \ \ \text{and} \ \ \phi\colon \pro/\gamma_{n-1}\lra \pro/\zeta_{n-1}.\]

We start with calculating $\Ez{s+t-1}$. The underlying diagram $\theta^*(X\wedge^{\bbL}Y)\in \Ho(\cM^{J_{s+t-1}})$ is
\begin{equation}\label{Jn1}
\xymatrix{
   &\ast&  &\ast&   &\ast&                   \\
 \ldots \ar[ur]  && A\wedge\tA \ar[ur] \ar[ul]  &&\ast \ar[ur] \ar[ul]   &&\ldots \ar[ul]
}
\end{equation}
where there only non-trivial entry is at $(\beta_{s-1},\beta_{t-1})$. From the diagram above we get
\begin{equation*}
\Ez{s+t-1}=\hocolim_{\pr/\zeta_{s+t-1}}(X\wedge^{\bbL}Y)\cong\hocolim_{J_{s+t-1}}\theta^*(X\wedge^{\bbL}Y)\cong \Sigma A\wedge \tA.
\end{equation*}
We do the same for $\Eg{s+t}, \Eg{s+t-1}$ and $\Eg{s+t-2}$. 
The value $\Eg{s+t-2}$ is the homotopy colimit of the diagram $X\wedge^{\bbL} Y\in \Ho(\cM^{\pr/\gamma_{s+t-2}})$ which is
\begin{equation*}\label{Egamman2}
\xymatrix{
&\ast& &A\wedge \tA& &A\wedge \tA& &\ast&  \\
\ldots \ar[ur]& &\ast \ar[ur] \ar[ul]& &A\wedge \tA \ar@{=}[ur] \ar@{=}[ul]& &\ast \ar[ur] \ar[ul]& & \ldots  \ar[ul]}
 \end{equation*}
with non-trivial entries at the places $(\beta_{s-1}, \beta_{t-1})$ on the bottom, $(\zeta_{s-1},\beta_{t-1})$ on the left and $(\beta_{s-1},\zeta_{t-1})$ on the right. Thus, $\Eg{s+t-2}\cong A\wedge \tA.$ Similarly, we have  
\begin{align*}
\Eg{s+t} &=  \hocolim_{\pr/\gamma_{s+t}}(X\wedge^{\bbL} Y)\cong * \\
\Eg{s+t-1}&= \hocolim_{\pr/\gamma_{s+t-1}}(X\wedge^{\bbL} Y)\cong \Sigma A\wedge\tA.
\end{align*}

\bigskip
\noindent\fbox{%
    \parbox{25em}{%
       {\bf Step 3: Calculating $\cone(\widehat{k}_{s+t})\lra \Sigma \Eg{s+t-1}$.}
    }%
}

\bigskip
We move on to calculate the map $\cone(\widehat{k}_{s+t})\lra S^1\otimes \Eg{s+t-1}$. From Definition \ref{elementary} we have the pushout square
\begin{equation*}
\xymatrix{ \Eg{s+t-1} \ar[r]^{\widehat{k}_{s+t}} \ar[d]  & \Ez{s+t}    \ar[d]    \ar@/^/[rdd]^{\ast} &     \\
          (I,0)\otimes \Eg{s+t-1}        \ar[r] \ar@/_/[rrd]_{\pi\wedge 1}      & \cone(\widehat{k}_{s+t}) \ar@{.>}[dr] &     \\
                                                   &                      & S^1\otimes \Eg{s-t-1} 
}
\end{equation*}
which, based on our computations, is the following.
\begin{equation*}
\xymatrix{
\Sigma A\wedge \tA \ar[d] \ar[r] & \ast         \ar[d] \ar@/^/[rdd]     & \\
  (I,0)\otimes \Sigma(A\wedge \tA) \ar@/_/[drr]_{\pi\wedge 1} \ar[r]  & \cone(\widehat{k}_{s+t}) \ar@{.>}[dr]  & \\    
                                  &                     & S^1\otimes (\Sigma A\wedge \tA)  }            
\end{equation*}
Recall from Proposition \ref{newconesE}, \eqref{hocolimofcone}, and Corollary \ref{conemonoidal} that there is a series of canonical isomorphisms
\[\cone(\widehat{k}_{s+t})\cong \cone(k_s\boxprod^{\bbL}\tk_t )\cong \cone(k_s)\wedge^{\bbL} \cone(\tk_t). \]
In our particular case, in which $k_s\colon A\to A$ and $\tk\colon \tA\to *$, this is
\[\cone(\widehat{k}_{s+t})\cong \cone(k_s\boxprod^{\bbL} k_t)\cong \Sigma^2A\wedge\tA \cong \Sigma A\wedge \Sigma\tA\cong\cone(k_s)\wedge^{\bbL} \cone(\tk_t).\]
This implies that the universal map out of the pushout is the identity map. Thus, the map 
\begin{equation*}
\cone(\widehat{k}_{s+t})\lra \Sigma \Eg{s+t-1}
\end{equation*} 
is the map $\Sigma A\wedge \Sigma \tA \to \Sigma^2 (A\wedge \tA)$ which is the composition of the canonical map (commutation of colimits) and the identity map. 

\bigskip
\noindent\fbox{%
    \parbox{25em}{%
       {\bf Step 4: Calculating $\widehat{l}_{s+t-1}\colon \Eg{s+t-1}\lra \Ez{s+t-1}$.}
    }%
}

\bigskip
From the posets above we can see 
directly that the map  
\begin{equation*}
\widehat{l}_{s+t-1}\colon \Eg{s+t-1}\lra \Ez{s+t-1}
\end{equation*} is the identity map induced by 
\begin{equation*}
\psi\colon \pr/\gamma_{s+t-1}\lra \pro/\zeta_{s+t-1}.
\end{equation*} 
Therefore the map 
\begin{equation*}
\Sigma\widehat{l}_{s+t-1}\colon \Eg{s+t-1}\lra \Sigma\Ez{s+t-1}
\end{equation*} 
is the identity map 
\begin{equation*}
1\colon \Sigma^2(A\wedge \tA)\lra \Sigma^2(A\wedge \tA).
\end{equation*}

\bigskip
\noindent\fbox{%
    \parbox{25em}{%
       {\bf Step 5: Calculating $\Ez{s+t-1}\lra \cone(\widehat{k}_{s+t-1})$.}
    }%
}

\bigskip
Lastly it remains to figure out the map 
\begin{equation*}
\Ez{s+t-1}\lra \cone(\widehat{k}_{s+t-1}).
\end{equation*} 
Recall from the proof of Proposition \ref{newconesE} that $\cone(\widehat{k}_{s+t-1})$ can be written as a homotopy colimit
\begin{equation*}
\cone(\widehat{k}_{s+t-1})\cong \hocolim_{\pr/\zeta_{s+t-1}}\left(\cone(\varepsilon_{X\wedge^{\bbL}Y} ) \right),
\end{equation*}
where $\phi\colon \pr/\gamma_{s+t-2}\to \pr/\zeta_{s+t+1} ,$ and $\varepsilon $ is the counit of the derived adjunction $(\bbL\phi_!,\phi^*)$. Pulling back the diagram 
$\cone(\varepsilon_{X\wedge^{\bbL}Y})$ to $J_{s+t-1}$ along the inclusion $\theta\colon J_{s+t-1}\to \proz{s+t-1},$ we obtain the diagram
\begin{equation*}
\begin{tikzpicture}[scale=1.7]
\node (zero) at (0,1) {$ \Sigma A\wedge \tA$};
\node (one) at (1,0) {$\ast$};
\node (two) at (2,1) {$\ast$};
\node (three)at ( 3,0) {$\ldots$ };
\node (four) at (-1,0) {$\ast$};
\node (five) at (-2,1) {$A\wedge \Sigma \tA   $};
\node (six) at (-3,0) {$\ldots$};
\draw [->] (one)-- (zero);
\draw [->] (one)-- (two);
\draw [->] (three)-- (two);
\draw [->] (four)-- (zero);
\draw [->] (four)-- (five);
\draw [->] (six)-- (five);
\end{tikzpicture}
\end{equation*}
with non-trivial entries at $(\zeta_{s-1},\zeta_t)$ and $(\zeta_{s},\zeta_{t-1})$ respectively. Recall the following diagram from \eqref{Jn1} 
$\theta^*(X\wedge^{\bbL}Y) \in \Ho(\cM^{J_{s+t-1}})$
\begin{equation*}
\begin{tikzpicture}[scale=1.7]
\node (zero)at (0,0) {$\ast$};
\node (one) at (1,1) {$\ast$};
\node (two) at (-1,1){$\ast$};
\node (three) at (2,0){$\ast$};
\node (four)  at (-2,0){$A\wedge\tA$};
\node (five) at  (3,1){$\ast$} ;
\node (six) at   (-3,1) {$\ast$};
\node (seven) at  (4,0) {$\ldots$};
\node (eight) at (-4,0) {$\ldots$};
\draw [->] (zero)-- (one) ;
\draw [->] (zero)-- (two);
\draw [->](three)-- (one);
\draw[->] (four)-- (two);
\draw [->](three)-- (five);
\draw [->](four)--(six);
\draw [->](seven)-- (five);
\draw [->](eight)--(six);
\end{tikzpicture}
\end{equation*}
with only non-trivial entry at $(\beta_{s-1},\beta_{t-1})$, left top being $(\zeta_{s-1},\zeta_t)$ and right top being $(\zeta_s,\zeta_{t-1})$.
Because of the shape of the underlying posets and the map, we can safely ignore 
the trivial entries, so the map $\Eg{s+t-1}\to\cone(\widehat{k}_{s+t-1})$
 can be taken as the map of homotopy pushouts
\[\hocolim(\ast \leftarrow A\wedge \tA \rightarrow \ast  ) \lra \hocolim(A\wedge \Sigma \tA  \leftarrow \ast \lra \Sigma A\wedge\tA   )    ,\]
induced by the following map of posets.
\begin{equation*}
\xymatrix{\ast  \ar[d] & \ar[l] A\wedge \tA  \ar[r] \ar[d] & \ast \ar[d]   \\
A\wedge \Sigma \tA   & \ar[l]    \ast \ar[r]   & \Sigma A\wedge\tA 
}
\end{equation*}
Now consider the above map of diagrams and the following map at the bottom.
\begin{equation*}
\xymatrix{ \ast  \ar[d] & \ar[l] A\wedge \tA  \ar[r] \ar[d] & \ast \ar[d]   \\
A\wedge \Sigma \tA  \ar[d]_{\tau} & \ar[l]  \ar[d]  \ast \ar[r]  \ar[d]  & \Sigma A\wedge\tA   \ar@{=}[d]   \\
  \Sigma A\wedge\tA                  &    \ar[l]    \ast \ar[r]                      & \Sigma A\wedge\tA    }
\end{equation*}
Here, $\tau$ is the map 
\[ A\wedge \Sigma \tA = A\wedge (S^1\wedge \tA ) \cong (A\wedge S^1) \wedge \tA \xrightarrow{\tau} (S^1\wedge A)\wedge \tA\cong \Sigma A\wedge \tA\] 
and the first map is the associativity 
isomorphism. By Lemma \ref{diagonal} the induced map 
of homotopy colimits is, up to weak equivalence, the diagonal map 
\[\operatorname{diag}\colon \Sigma A\wedge \tA\lra (\Sigma A\wedge \tA) \vee (\Sigma A\wedge \tA).\] 
Hence, the map \eqref{compositionforiE} is up to weak equivalence 
the diagonal map but with a sign introduced by the twist map as above. This  and Corollary \ref{homoofnewcones} imply that that indeed the differential
\[d\colon C^{(s+t)}(i^*E)\lra C^{(s+t-1)}(i^*E)\] coincides with the differential of the tensor product of
\[\left((D^sL^s) \ox (D^t\widetilde{L}^t) \right)^{s+t} \lra \left((D^sL^s) \ox (D^t\widetilde{L}^t) \right)^{s+t-1}.\]

\bigskip
\noindent\fbox{%
    \parbox{20em}{%
       {\bf Step 6: $\cQ(i^*E)^{s+t-1} \longrightarrow \cQ(i^*E)^{s+t-2}$.}
    }%
}

\bigskip
We do not need to do any extra work to determine the other differential, namely to check that the differential
\[ C^{(s+t-1)}(i^*E)\lra C^{(s+t-2)}(i^*E)  \] 
coincides with the differential
\[\left((D^sL^s) \ox (D^t\widetilde{L}^t) \right)^{s+t-1} \lra \left((D^sL^s) \ox (D^t\widetilde{L}^t) \right)^{s+t-2},\]
since by construction $(C_*(i^*E),d)$ is a differential graded object and that means that by necessity $d[1]\circ d=0$ on $C_*(i^*E)$. This concludes the proof.
\end{proof}

To conclude this section, by combining Corollary \ref{homoofnewcones} and Proposition \ref{differentialsdisks} we have proved the following proposition.
\begin{prop}\label{propB}
Let $X,Y\in \cL$ such that $F_*(X_{{\alpha}_n}), F_*(Y_{{\alpha}_n})\in  \cAp$ for any $n\in \bbZ/N\bbZ$ and any $\alpha\in \left\{\beta,\zeta\right\}$. There is a natural isomorphism
\[
\cQ(i^*\bbL\pr_!(X\wedge^{\bbL}Y ))\cong \cQ(X)\otimes \cQ(Y) .
\]
\end{prop}
\qed

\subsection{Technical lemmas}
In this subsection we prove two technical lemmas that are used in the the previous proofs. 
The first lemma is about the canonical map from the suspension of an object to the wedge product of suspensions in a stable, simplicial model category $\cM$. 
The second lemma is about pushout-products of injective morphisms in a hereditary abelian category $\cA$.

\begin{lemma}\label{diagonal}
Let $\cM$ be a stable simplicial model category and	let $X\in \cM$. Consider the following map of homotopy pushouts
	\[\hocolim(*\leftarrow X\lra *) \lra \hocolim(\Sigma X\leftarrow *\lra \Sigma X  ) .  \]
	Then the above map is, up to isomorphism in $\Ho(\cM)$, the diagonal map
	\[\operatorname{diag}\colon \Sigma X\lra \Sigma X \vee \Sigma X.  \]
\end{lemma}

\begin{proof}
	Let $CX= (I,0)\otimes X $ be the cone of $X$ and let $i\colon X\lra CX$ be the canonical inclusion, which is a 
	cofibration. We choose a model for $\Sigma X$ as the homotopy pushout  
	\[
	\Sigma X\cong \hocolim(CX \leftarrow X \lra CX).\] 
	In fact, we can take this to be the ordinary pushout $\colim(CX \longleftarrow X\lra CX )$ since $i\colon X\lra CX$ is a cofibration. 
	From this model we get directly that the induced map on pushouts
	\[\begin{tikzcd}
		CX \arrow[swap]{d}{\pi\otimes 1} & \arrow[swap]{l}{i}  X \arrow{r}{i} \arrow{d} & CX  \arrow{d}{\pi\otimes 1}   \\
		\Sigma X               & \arrow{l}  \ast  \arrow{r}                    & \Sigma X
	\end{tikzcd}
	\]
	where $\pi \colon I\to S^1 $ is the projection is indeed the diagonal map $\mathrm{diag}\colon \Sigma X\lra \Sigma X \vee \Sigma X$. Hence, the induced map of homotopy pushouts is the diagonal map up to natural isomorphism.
\end{proof}

\begin{lemma}\label{ppinjective}
	Let $\cA$ be a heridary abelian category. Let $X,Y, U ,V\in \cAp$  and let $f\colon X\lra Y$ and $g\colon U\lra V$ be injective maps. Then the pushout-product map $f\boxprod g $ is injective.
\end{lemma}
\begin{proof}
	Since $g\colon U\lra V$ is monomorphism we have the short exact sequence
	\[ 0\lra U\xrightarrow{g} V \stackrel{j}{\rightarrow} \operatorname{coker}g \lra 0.\]
	Notice that the dimension of the abelian category $\E$-modules is $1$, which implies that $\operatorname{coker}g$ is a projective module since it is a submodule of $V$. Since $X$ is flat, $X\otimes -$ is an exact functor, 
	which means that the sequence
	\[0\lra  X\otimes U \xrightarrow{1\otimes g} X\otimes V \xrightarrow{1\otimes j} X\otimes \operatorname{coker}g \lra 0   \]
	is short exact.
	Consider the diagram
	\[ 
	\xymatrix{ 0 \ar[r]  & X\otimes U \ar[r]^{1\otimes g} \ar[d]_{f\otimes 1}  & X\otimes V \ar[r]^{1\otimes j} \ar[d] & X\otimes \coker g \ar[r] \ar@{=}[d] & 0 \\
		0 \ar[r]  & Y\otimes U \ar[r] \ar@{=}[d]  &  P  \ar[r] \ar[d]_{f\boxprod g}       &  X\otimes \coker g\ar[r] \ar[d]^{f\otimes 1} & 0 \\
		0 \ar[r]  &   Y\otimes U \ar[r]_{1\otimes g}   & Y\otimes V \ar[r]_{1\otimes j}   &  Y\otimes \coker g  \ar[r]  & 0
	}
	\]
	where $P$ is the pushout of $1\otimes g$ and $g\otimes 1$. Since the top left square is cocartesian, the canonical map $\coker(1\otimes g) \stackrel{\cong}{\rightarrow} \coker(Y\otimes U\lra P)$ is an isomorphism, so the middle row is also 
	exact. Now note that the morphism $f\otimes 1\colon X\otimes \coker g\lra Y\otimes \coker g $ is injective since $\coker g$ is projective. Applying the snake lemma gives us 
	that $f\boxprod g $ is a monomorphism.

\end{proof}

\section{Main Result}\label{sec:mainresult}
\subsection{Homotopy Colimit Calculations}\label{sec:hocolimcalcs}
In this section we discuss how the functor $i^*\bbL\pr_!$ interacts with the homotopy colimits of the various diagram categories, giving us the right hand side of the main diagram \ref{bigdiagram}. The main result of the section is the following.
\begin{thm}\label{theoremB}
For any pair of diagrams $(X,Y)\in \Ho(\cM^{\cC_N})\times \Ho(\cM^{\cC_N})$, the homotopy colimit of the diagram $i^*\bbL\pr_!(X\wedge^{\bbL}Y)\in \Ho(\cM^{\cC_N})$ is naturally isomorphic to the smash 
product of the homotopy colimits of $X$ and $Y$, that is,
\[
\hocolim_{\cC_N}\left(i^*\bbL\pr_!(X\wedge^{\bbL}Y)\right)\cong \hocolim_{\cC_N}X \wedge^{\bbL}  \hocolim_{\cC_N}Y.
\]
\end{thm}
Recall that the functor 
\[i^*\bbL\pr_!(-\wedge^{\bbL}-)\colon \Ho(\cM^{\cC_N})\x\Ho(\cM^{\cC_N})\to \Ho(\cM^{\cC_N})\] 
is the composition
\[\Ho(\cM^{\cC_N})\x \Ho(\cM^{\cC_N}) \xrightarrow{ \wedge^{\bbL}}  \Ho(\cM^{\cC_N\times \cC_N})\xrightarrow{\bbL\pr_! }  \Ho(\cM^{\cD_N}) \xrightarrow{i^*} \Ho(\cM^{\cC_N}).\]
In order to prove Theorem \ref{theoremB} we will break it apart into smaller pieces. Consider the following diagram.
\begin{equation*}\label{dexiomeros}
\xymatrix{ \Ho(\cM^{\cC_N})\times \Ho(\cM^{\cC_N})    \ar[r]  \ar[d]_{\wedge^{\bbL}}                            &  \Ho(\cM) \\
          \Ho(\cM^{\cC_N\times \cC_N})  \ar[d]_{\bbL\pr_!}    \ar[ur]           & \\
					\Ho(\cM^{\cD_N})    \ar[d]_{i^*}       \ar[uur]                   & \\
					\Ho(\cM^{\cC_N})  \ar@/_/[uuur] 
}
\end{equation*}
The top horizontal functor is the smash product of homotopy colimits of crowned diagrams, that is, $\hocolim_{\cC_N}X \wedge^{\bbL}\hocolim_{\cC_N}Y.$
The three other functors are the homotopy colimit functors

\begin{enumerate}[(i)]
\item $\hocolim_{\cC_N\x \cC_N}\colon  \Ho(\cM^{\cC_N\times \cC_N}) \lra \Ho(\cM)$, \\
\item $\hocolim_{\cD_N}\colon\Ho(\cM^{\cD_N})\lra \Ho(\cM),   $ \\
\item $\hocolim_{\cC_N}\colon \Ho(\cM^{\cC_N}) \lra \Ho(\cM)$.
\end{enumerate}

Theorem \ref{theoremB} asserts that the outer triangle above commutes up to isomorphism. This will follow once we show that all the small triangles commute up to isomorphism.

\begin{lemma}\label{lemma1}
The top triangle and the middle triangle commute, that is, 
\[\hocolim_{\cC_N}X\wedge^{\bbL}\hocolim_{\cC_N}Y\cong \hocolim_{\cC_N\times \cC_N}(X\wedge^{\bbL}Y) \]
and
\[\hocolim_{\cC_N\times \cC_N}(X\wedge^{\bbL} Y)  \cong \hocolim_{\cD_N} \pr_!(X\wedge^{\bbL} Y).\]
\end{lemma}
\begin{proof}

The first assertion follows from Corollary \ref{hocolimIJ} as a direct application for $\cC=\cD=\cC_N$. The second assertion follows from the fact that the homotopy colimit of a homotopy left Kan 
extension of a diagram is isomorphic to the homotopy colimit of the diagram itself \cite[Proposition 4.3.2]{RI20}.
\end{proof}

We will prove the above proposition by proving that the functor $i\colon \cC_N\lra \cD_N$ using the following definition, see \cite[Definition 8.5.1]{Riehl}.

\begin{defn}\label{def:homotopyfinal} A functor between small categories $K\colon\cC\lra \cD$ is \emph{homotopy final} 
(or \emph{homotopy terminal}) if for every object $d\in \cD$, the simplicial set $N(d/K)$ is contractible . 
\end{defn}

A convenient way to check whether a poset is contractible 
is given by Quillen \cite[Section 1.5]{QU78}: a poset $\cC$ is \emph{conically contractible} if there is an object $c_0\in \cC$ and a map of posets $f\colon \cC\to \cC$ such that $c\leq f(c) 
\geq c_0$ for every $c\in \cC$. In this 
case one can show that the identity $1_{\cC}$, the map $f$, and the constant map with value $c_0$ from $\cC$ to itself are homotopic (that is to say, their realizations are homotopic), and hence $\cC$ 
is contractible. So, given a diagram $E\in \Ho(\cM^{\cD_N})$, to check that the canonical morphism
\begin{equation*}
\phi_i\colon \hocolim_{\cC_N}i^*E \lra \hocolim_{\cD_N}E
\end{equation*}
is an isomorphism it suffices to check that the slice categories $\alpha_n/i$ of the functor $i\colon \cC_N\to \cD_N$ are contractible for any $\alpha\in \left\{\zeta,\gamma,\beta\right\}$ and any $n\in \bbZ/N\bbZ$. 

We will now apply this to our functor $i\colon \cC_N\lra \cD_N$, which is the inclusion of the two-row crowned diagram into the three-row crowned diagram  \eqref{definitioni}.

\begin{lemma}
The functor $i\colon \cC_N\lra \cD_N$ is homotopy final.
\end{lemma}

\begin{proof}
We will prove the above proposition by applying Quillen's criterion of conical contractible posets. First, we identity the slice categories $\zeta_n/i, \gamma_n/i$ and $\beta_n/i$ and then we will 
check that they are indeed conically contractible. We start with $\zeta_n/i$. By definition, 
\[\zeta_n/i=\left\{\alpha_n\in \cC_N\colon i(\alpha_n)\geq \zeta_n \right\}= \left\{\zeta_n\right\}.\]
Since this poset contains only one element it is is obviously contractible. The next slice categories are of the form $\gamma_n/i$. By definition,
\begin{equation*}
\gamma_n/i=\left\{\alpha_n\in \cC_N\colon i(\alpha_n)\geq \gamma_n\right\},
\end{equation*}
that is, $\gamma_n/i$ is the poset
\begin{equation*}\xymatrix{ \zeta_{n} & \zeta_{n+1}. \\
                             \beta_n \ar[u] \ar[ur]  &   }
\end{equation*}
We choose $\beta_n$ and $1\colon \gamma_n/i\lra \gamma_n/i$. Directly from above we can see that $\gamma_n/i$ is conically contractible.
The last case are the slices $\beta_n/i$. By definition,
\begin{equation*}\beta_n/i=\left\{\alpha_n\in\cC_N\colon i(\alpha_n)\geq \beta_n\right\},
\end{equation*}
which is the poset
\[\xymatrix{\zeta_{n}  & \zeta_{n+1} \\
           \beta_{n} \ar[u] \ar[ur]     & \beta_{n+1}. \ar[u]   }\]
We choose $\beta_n$ and the map of $\beta_n/i\lra \beta_n/i$ as follows.
\[\zeta_{n}\mapsto \zeta_{n}, \zeta_{n+1} \mapsto \zeta_{n+1}, \beta_{n}\mapsto \beta_{n}, \beta_{n+1}\mapsto \zeta_{n+1}\] 
With these choices, we can see that 
the poset $\beta_n/i$ is conically contractible.
\end{proof}

Finally, we obtain the commutativity of the bottom triangle of our big diagram, which also concludes the proof of Theorem \ref{theoremB}. Recall the functor $i\colon \cC_N\to \cD_N$ from
\begin{cor}\label{lemma3}
The bottom triangle of (\ref{bigdiagram}) commutes, that is, 
\begin{equation*}
\hocolim_{\cC_N}i^*E\cong \hocolim_{\cD_N}E.
\end{equation*}
\qed
\end{cor}

\subsection{Proof of Main Theorem}

Finally, we are in a position to assemble all our work into our main theorem. 

\begin{thm}\label{thm:main}
Let $\cA$ be a hereditary abelian category, and $\cM$ be a monoidal stable model category such that Franke's functor
\[
\cR\colon (\dtwc,\otimes^{\bbL})\lra (\Ho(\cM),\wedge^{\bbL})
\]
exists and is an equivalence.
Then $\cR$ preserves the monoidal products up to a natural isomorphism, that is, 
\[\cR(M_* \otimes^{\bbL} M_*) \cong \cR( M_*) \wedge^{\bbL} \cR(M_*).   \]
\end{thm}

\begin{proof}
We assemble our proof along the lines of the diagram \ref{bigdiagram}.
Let $M_*$ and $N_*$ be objects in $\dtwc$. By Convention \ref{homotopycategoryofM}, both objects are cofibrant.
Since $M_*$ is cofibrant, the functor 
\[M_*\otimes -\colon  \twc   \lra  \twc  \] 
is left Quillen, see \cite[Remark 4.2.3]{Hovey}, which means it preserves cofibrant objects. Since both objects are cofibrant, the tensor product 
$M_*\otimes N_*$ represents the derived tensor 
product in $(\dtwc,\otimes^{\bbL})$ and in particular it also cofibrant. Recall from Example \ref{twistedcomplexes} that the cofibrant objects in $\twc$ are the projective objects in $\cA$. This 
means, in particular, that
$M_*, N_*$ and $M_*\otimes N_*$ all belong to $\cAp$. We recall some notation from Section \ref{sec:monoidalQ}. Given a crowned diagram $X\in \cM^{\cC_N}$ as follows
\begin{equation*}\xymatrix{  &X_{\zeta_0}  &  X_{\zeta_1}  & \ldots & X_{\zeta_{N-1}}   \\ 
                             &X_{\beta_0} \ar[u] \ar@{.>}[urrr] & \ar[u]  \ar[ul] X_{\beta_1}   & \ar[ul] \ldots & \ar[u] \ar[ul] X_{\beta_{-1}} }
	\end{equation*}
we set
\[Z^{(n)}(X)=F_*(\Xz{n}),\  B^{(n)}(X)=F_*(\Xb{n}),\  C^{(n)}(X)= F_*(\cone(\Xb{n-1}\lra \Xz{n})) .\]
Given $(M_*,d)\in \twc$ one can construct a crowned diagram $X$ in $\cL$ such that 
\begin{align*}
(C_*(X),d)\cong(M_*,d)\\
Z_*(X)\cong\ker d \\
B_*(X)\cong \operatorname{im} d.
\end{align*} 
By the discussion above, $M_*\in \cAp $. By assumption, $\cA$ is a hereditary abelian category, in other words, 
$\operatorname{gl.dim} \cA =1 $. This implies that $\operatorname{ker} d, \operatorname{im} d \in \cAp$ since they are submodules of $M_*$.

 Hence, for the crowned diagram $X\cong \cQ^{-1}(M_*)$ we have $F_*(\Xa{n})\in \cAp$ for every $n\in \bbZ/N\bbZ$ and any $\alpha\in \left\{\zeta,\beta\right\}$. Similarly, 
for the dg-object $(N_*,d)$ we get a crowned diagram $Y\cong \cQ^{-1}(N_*)$ such that $F_*(\Ya{n})\in \cAp$ for every $n\in \bbZ/N\bbZ$ and any $\alpha\in \left\{\zeta,\beta\right\}$. 

Now, by Theorem \ref{theoremA},
\[ \cQ(i^*\bbL\pr_!(X\wedge^{\bbL}Y))\cong \cQ(X)\otimes \cQ(Y)=M_*\otimes N_*  \]
and by Theorem \ref{theoremB},
\[
\hocolim_{\cC_N}\left(i^*\bbL\pr_!(X\wedge^{\bbL}Y)\right)\cong \hocolim_{\cC_N}X \wedge^{\bbL}  \hocolim_{\cC_N}Y.
\]
Finally, we recall that Franke's realization functor \eqref{recfunctor} is defined by
\[\cR= \hocolim_{\cC_N}\circ \cQ^{-1},\]
which concludes the proof.

\end{proof}

The assumptions of Theorem \ref{thm:main} are satisfied in the following instances.

\begin{example}
From \cite[Corollary 5.2.1]{PA12} we know that 
\[
\cR: \cD(\pi_*R) \longrightarrow \cD(R)=\Ho(R\mbox{-mod})
\]
is an equivalence for a ring spectrum $R$ with $\pi_*(R)$ concentrated in degrees that are multiples of some $N>1$ and global dimension of $\pi_*(R)$ equal to 1. This satisfies the assumption of our Theorem \ref{thm:main} and applies to $R= KU$, $R=KU_{(p)}$, $R=E(1)$ (complex $K$-theory), and $R=k(n)$ (connective Morava $K$-theory). 
\end{example}

\begin{example}
By \cite{Fran} and Roitzheim \cite{Ro} we know that
\[
\cR: \dtwc \longrightarrow \Ho(L_1 \mathcal{S})
\]
is an equivalence. Here, $\cA$ is the category of $E(1)_*E(1)$-comodules, and $L_1 \mathcal{S}$ is a suitable category of spectra equipped with the $K$-local model structure at an odd prime. 
Note that as mentioned in Example \ref{twistedcomplexes}, that while $\cA$ does not have enough projectives, all our proofs also work when working with comodules whose underlying $E(1)_*$-module is projective, see also the first author's thesis \cite{NN}.
\end{example}


\end{document}